\documentclass[journal,twocolumn,10pt]{IEEEtran}
\usepackage{xcolor}
\usepackage{graphicx}
\usepackage{amssymb}
\usepackage{amsmath}
\usepackage{amsthm}
\usepackage{mathtools}
\usepackage{dsfont}
\usepackage{cite}
\usepackage{stfloats}
\usepackage{subfigure}
\usepackage{psfrag}
\usepackage[mathscr]{euscript}
\usepackage{acronym}    
\usepackage{booktabs}
\usepackage{tablefootnote}
\usepackage{optidef}
\usepackage{stfloats}
\usepackage{algorithm}
\usepackage{algorithmic}  
\usepackage{multirow}
\usepackage{caption}
\usepackage{epstopdf}
\usepackage{comment}

\usepackage{float}
\usepackage{balance}
\usepackage{subfigure}  

\def\cG{{\cal G}}

\def\cO{{\cal O}}

\def\cS{{\cal S}}

\acrodef{GMRF}{Guassian Markov Random Field}
\acrodef{CCDF}{complementary cumulative distribution function}
\acrodef{CF}{characteristic function}
\acrodef{PPP}{Poisson point processe}
\acrodef{RV}{random variable}
\acrodef{i.i.d.}{independent and identically distributed}
\acrodef{PDF}{probability distribution function}
\acrodef{CDF}{cumulative distribution function}
\acrodef{ch.f.}{characteristic function}
\acrodef{AWGN}{additive white Gaussian noise}
\acrodef{SNR}{signal-to-noise ratio}
\acrodef{LRT}{likelihood ratio test}
\acrodef{DRT}{distance ratio test}
\acrodef{GLRT}{generalized likelihood ratio test}
\acrodef{CRLB}{Cram\'{e}r-Rao lower bound}
\acrodef{CRB}{Cram\'{e}r-Rao bound}
\acrodef{ZZLB}{Ziv-Zakai lower bound}
\acrodef{ZZB}{Ziv-Zakai bound}
\acrodef{LOS}{line-of-sight}
\acrodef{ToF}{time-of-flight}
\acrodef{NLOS}{non-line-of-sight}
\acrodef{GDOP}{geometric dilution of precision}
\acrodef{GPS}{Global Positioning System}
\acrodef{FIM}{Fisher information matrix}
\acrodef{PEB}{position error bound}
\acrodef{SPEB}{squared position error bound}
\acrodef{TOA}{time-of-arrival}
\acrodef{TOF}{time-of-flight}
\acrodef{WSN}{wireless sensor network}
\acrodef{MAC}{medium access control}
\acrodef{RSS}{received signal strength}
\acrodef{WAF}{wall attenuation factor}
\acrodef{TDOA}{time difference-of-arrival}
\acrodef{RF}{radiofrequency}
\acrodef{RTT}{round-trip time}
\acrodef{AOA}{angle-of-arrival}
\acrodef{MF}{matched filter}
\acrodef{ED}{energy detector}
\acrodef{ML}{maximum likelihood}
\acrodef{MSE}{mean-square error}
\acrodef{RMSE}{root-mean-square error}
\acrodef{LEO}{localization error outage}
\acrodef{ppm}{part-per-million}
\acrodef{ACK}{acknowledge}
\acrodef{UWB}{Ultrawide bandwidth}
\acrodef{TNR}{threshold-to-noise ratio}
\acrodef{LS}{least squares}
\acrodef{IR-UWB}{impulse radio UWB}
\acrodef{FCC}{Federal Communications Commission}
\acrodef{TH}{time-hopping}
\acrodef{PPM}{pulse position modulation}
\acrodef{MUI}{multi-user interference}
\acrodef{PDP}{power delay profile}
\acrodef{BPZF}{band-pass zonal filter}
\acrodef{SIR}{signal-to-interference ratio}
\acrodef{SINR}{signal-to-interference-plus-noise ratio}
\acrodef{RFID}{radio frequency identification}
\acrodef{WPAN}{wireless personal area network}
\acrodef{WWB}{Weiss-Weinstein bound}
\acrodef{DP}{direct path}
\acrodef{MF}{matched filter}
\acrodef{MMSE}{minimum-mean-square-error}
\acrodef{SBS}{serial backward search}
\acrodef{SBSMC}{serial backward search for multiple clusters}
\acrodef{NBI}{narrowband interference}
\acrodef{WBI}{wideband interference}
\acrodef{INR}{interference-to-noise ratio}
\acrodef{CR}{channel response}
\acrodef{CIR}{channel impulse response}
\acrodef{CR}{channel  response}
\acrodef{RADAR}{radar}
\acrodef{MUR}{Multistatic radar}
\acrodef{JBSF}{jump back and search forward}
\acrodef{HDSA}{high-definition situation-aware}
\acrodef{RRC}{root raised cosine}
\acrodef{ST}{simple thresholding}
\acrodef{BTB}{Bellini-Tartara bound}
\acrodef{P-Max}{$P$-Max}  
\acrodef{MIMO}{multiple-input multiple-output}
\acrodef{MAP}{maximum a posteriori}
\acrodef{FG}{factor graph}
\acrodef{OP}{outage probability}
\acrodef{WED}{wall extra delay}
\acrodef{RMS}{root mean square}
\acrodef{SPAWN}{sum-product algorithm over a wireless network}
\acrodef{MDD}{minimum distance distribution}
\acrodef{MAP}{maximum a posteriori probability}
\acrodef{SAP}{small cell access point}
\acrodef{UE}{user equipment}
\acrodef{MBS}{macro cell base station}
\acrodef{UER}{\ac{UE} Relay}
\acrodef{D2D}{device-to-device}
\acrodef{MBS}{macro base station}
\acrodef{CSI}{channel state information}
\acrodef{OGR}{outage guard region}
\acrodef{FUR}{feasible UER region}
\acrodef{EHR}{energy harvesting region}
\acrodef{EH}{energy harvesting}
\acrodef{D2D-EHSN}{D2D communication provided \ac{EH} small cell network}
\acrodef{D2D-EHHN}{D2D communication provided \ac{EH} heterogeneous network}
\acrodef{3GPP}{3rd Generation Partnership Project}
\acrodef{BS}{base station}
\acrodef{DF}{decode and forward}
\acrodef{CCDF}{complementary cumulative distribution function}
\acrodef{ZF}{zero forcing}
\acrodef{RZF}{regularized zero forcing}
\acrodef{WLLN}{weak law of large number}
\acrodef{SLLN}{strong law of large numbers}
\acrodef{TDD}{Time-division duplex}
\acrodef{EE}{energy efficiency} 
\acrodef{HetNet}{heterogeneous network} 
\acrodef{SCP}{Single Cell Processing}
\acrodef{CBF}{Coordinated Beamforming}
\usepackage{color}
\usepackage{dsfont}
\usepackage{bbm}

\DeclareMathAlphabet{\mathsf}{OML}{cmbr}{m}{it}

\newtheorem{theorem}{\bf Theorem}
\newtheorem{lemma}{\bf Lemma}

\DeclareMathOperator{\tr}{\mathrm{tr}}

\newcommand{\bd}{\begin{description}}
\newcommand{\ed}{\end{description}}
\newcommand{\be}{\begin{enumerate}}
\newcommand{\ee}{\end{enumerate}}
\newcommand{\bi}{\begin{itemize}}
\newcommand{\ei}{\end{itemize}}
\newcommand{\bl}{\begin{list}}
\newcommand{\el}{\end{list}}
\newcommand{\bt}{\begin{tabbing}}
\newcommand{\et}{\end{tabbing}}

\newcommand{\paperTitle}{Joint Network Topology Inference via Structured Fusion Regularization}

\begin{document}

{
\title{\paperTitle}
\author{

	    Yanli~Yuan, \textit{Student Member, IEEE,}         
        De Wen~Soh,
        Xiao~Yang,  
        Kun~Guo, \textit{Member, IEEE,}\\
        and Tony~Q.~S.~Quek, \textit{Fellow, IEEE}
       
\thanks{Yanli~Yuan, De~Wen~Soh, Kun~Guo and Tony~Q.~S.~Quek are with the Singapore University of Technology and Design, Singapore (e-mail:~yanli\_yuan@mymail.sutd.edu.sg, dewen\_soh@sutd.edu.sg,  kun\_guo@sutd.edu.sg,  and tonyquek@sutd.edu.sg).}
\thanks{ Xiao Yang is with the State Key Laboratory
of Integrated Service Institute of Information Science, Xidian University, Xi’an, Shaanxi, 710071, China (e-mail:XYang\_2@stu.xidian.edu.cn.)}

    }
\maketitle
\acresetall
\thispagestyle{empty}
\begin{abstract}
Joint network topology inference represents a canonical problem of jointly learning multiple graph Laplacian matrices from heterogeneous graph signals. In such a problem, a widely employed assumption is that of a simple common component shared among multiple graphs. However, in practice, a more intricate topological pattern, comprising simultaneously of \textit{homogeneity} and \textit{heterogeneity} components, would exhibit in multiple graphs. In this paper, we propose a general graph estimator based on a novel structured fusion regularization that enables us to jointly learn multiple graph Laplacian matrices with such complex topological patterns, and enjoys both high computational efficiency and rigorous theoretical guarantees. Moreover, in the proposed regularization term, the topological pattern among graphs is characterized by a Gram matrix, which enables our graph estimator to flexibly model different types of topological patterns through different Gram matrix choices. Computationally, the regularization term, coupling the parameters together, makes the formulated optimization
problem intractable, and thus, we develop an algorithm based on the alternating direction method of multipliers (ADMM) to solve it efficiently. Theoretically, non-asymptotic statistical analysis is provided, which proves the consistency of the joint graph estimator and illustrates the impact of the data sample size and the graph structures as well as the topological patterns on the estimation accuracy. Finally, the superior performance of the proposed method is demonstrated through simulated and real data examples.

%
\begin{IEEEkeywords}
Graph learning, heterogeneous graph signals,
joint network topology inference, structured fusion regularization, topological
patterns.
\end{IEEEkeywords}

\end{abstract}

\acresetall

\section{Introduction}
Modern data analysis tasks typically involve large sets of structured data that reside on real-world networks such as social networks, wireless communication networks, and transportation networks \cite{Acemap209872764,angles2008survey}. Graphs are appealing mathematical tools used in various fields to represent and analyze the said structured data (often conceptualized as graph signals). For example, in the field of graph signal processing (GSP) \cite{shuman2013gsp,8347162}, many models such as filtering\cite{shuman2020localized}, wavelet transformation \cite{hammond2011wavelets} exploit the graph topology to effectively process graph signals. However, in many practical cases, nodal observations (i.e., measurements of graph signals) are unorganized, and the underlying graph topology might (often) not be explicitly available. Therefore, the first key step is to use the observed data to infer or learn the underlying graph topology, hence facilitating downstream data analysis tasks.

To learn the underlying graph topology from data, many graph learning methods have been developed, see recent review papers \cite{8700665,8700659}. The focus so far in the literature has been on learning a \textit{single} graph Laplacian matrix, by assuming a signal model linking all observations to a single unknown graph \cite{7953415,8309394,egilmez2018graph,7979524,kalofolias2016learn,yuan2019learning}. However, in many applications, we often need to deal with heterogeneous graph signals that come from several networks defined over the same node set, but with possibly different sets of edges.  Examples of such heterogeneous graph signals include gene expression data from subjects with different stages of the same disease \cite{gene}, webpages collected from a university department with different categories corresponding to the faculty, student, course \cite{Cravenweb}, user profiles observed from social networks where
the same set of users can have different types of social interactions \cite{bindu2017discovering}, to name just a few. In these cases, it is more realistic to jointly infer multiple graphs so as to improve the estimation efficiency by exploiting common statistics of data from different classes \cite{navarro2020joint, 8335493}. 

Nevertheless, joint inference of multiple graphs is challenging mainly because it requires the joint graph estimator to be powerful enough to capture the complex topological patterns among networks, while also being general enough to discover various types of topological patterns in different scientific fields. On one hand, this is due to the fact that real-world networks are usually complex, and thus an intricate topological pattern, comprising simultaneously of \textit{homogeneity} and \textit{heterogeneity} components, would exhibit in multiple related networks. For example, the conditional dependence relationships among words in webpages are expected to share commonality across different categories (webpage homogeneity) while maintaining some levels of uniqueness within each category (webpage heterogeneity)\cite{Cravenweb}. Hence, if the joint graph estimator is able to capture complex topological patterns among related networks, it could better exploit additional sources of information and, thereby resulting in more accurate topology estimates. On the other hand, in practice, the multiple networks of interest can possibly be indexed by time, different subjects, demographic variables, or sensing modalities, accordingly, the characteristics of the topological patterns would vary with applications \cite{boccaletti2006complex,newman2001structure}. For example, the case in genomics \cite{gene}, where the gene regulatory networks for different patients are almost identical except that a handful of edges vary among them (see Fig. \ref{figure:illustrationstructuralpattern}(a) as an illustration); and the case in social networks \cite{bindu2017discovering}, where the friendship connections of some people would randomly change at different time windows (illustrated in Fig.\ref{figure:illustrationstructuralpattern}(b)).
Therefore, for joint network topology inference, it is useful to seek graph estimators that are generic, such that they can be flexibly applied to learn graphs with different types of complex topological patterns for different applications. Unfortunately, the existing
researches always employ an assumption of a simple common
component shared among multiple networks, e.g., a smoothly
time-varying topological pattern \cite{7953415Kalofolias,baingana2016tracking, 2020Time,cardoso2020learning}, regardless of many
other desired types of intricate topological patterns in different
applications. 

\begin{figure*}[t] 
\setlength{\belowcaptionskip}{-0.5cm}
\centering
\subfigure[]{
\includegraphics[width=0.9 \columnwidth , height=0.27 \columnwidth]{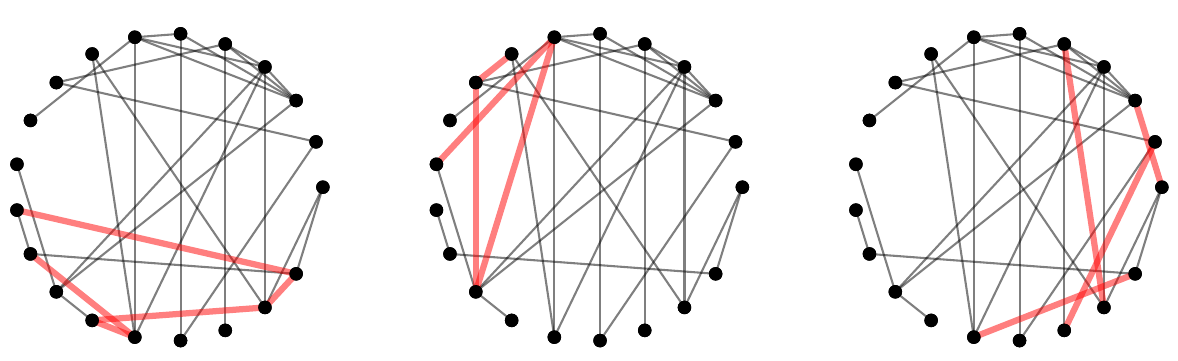}
}
\hspace{0.2ex}
\subfigure[]{
\includegraphics[width=0.9 \columnwidth, height=0.25 \columnwidth]{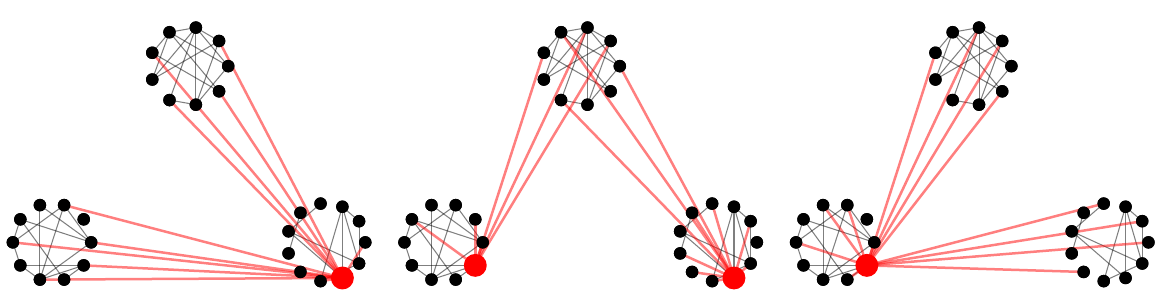}
}

\setlength{\belowcaptionskip}{-0.2cm}
\caption{The illustration of two types of complex topological patterns among three networks. The black lines are the edges shared in all three networks, while the red lines represent the unique edges. Generally, in both two cases, the topological patterns are complex, consisting of both heterogeneity and homogeneity components among the three networks. However, the characteristics of the two topological patterns are different. (a) The three networks are almost identical except that a handful of edges varies among them. (b) There include some nodes randomly rewiring most of their edges in each network. }
\label{figure:illustrationstructuralpattern}
\end{figure*}

In this paper, to address the above challenges, we introduce a powerful and general graph learning framework to jointly learn multiple graph Laplacian matrices from high-dimensional heterogeneous
graph signals. Our framework can simultaneously achieve three goals: (1) modelling various types of complex topological patterns among networks flexibly, (2) computational efficiency, and (3) theoretical
guarantees. More detailedly, the contributions are as follows: 

First, we propose a novel graph estimator based on a structured fusion regularization. Such a regularization term helps incorporate the topological patterns among networks into the learning process, which  endows the graph estimator the power of capturing complex network topologies, thereby improving the estimation accuracy. Specifically, the proposed regularization term combines a $\ell_1$-norm with a weighted $\ell_2$-norm. On one hand, the $\ell_1$-norm encourages the sparsity in each learned
graph Laplacian matrix. On the other hand, the weighted $\ell_2$-norm represents a family of group-structured norms \cite{negahban2009unified}, each of which is constructed by a Gram matrix. The Gram matrix typically reflects the in-group and across-group correlations between the edges of different graphs, which has a natural tendency of fusing each group of edges according to their correlations. Thus, the Gram matrix can be chosen depending on the context to model the specific topological pattern needed for network analytics. Naturally, different choices of the Gram matrix can induce different topological patterns, which leads to great flexibility in learning various types of networks in different fields. In this paper, we specify three choices of the Gram matrix as instances: e.g., the group graph lasso \cite{8347160}, time-varying graph lasso \cite{2020Time}, and Laplacian shrinkage penalty \cite{huang2011} (see Section \ref{sec:proposedestimator} for details).  In addition, our estimator can obviate the curse of
dimensionality in high-dimensional settings, since it's able to improve statistical efficiency through fusing information from related graphs. 

Second, we develop an algorithm based on the Alternating Direction Method of Multipliers (ADMM) \cite{2010Distributed} to compute the proposed graph estimator. This is necessary because the structured fusion regularization term couples the parameters together, even though efficient methods exist to solve the single graph learning problem \cite{8747497, 9414791}, no ready-to-use methods exist for our coupled multiple cases. To develop the ADMM solver, we decouple the parameters fused in the regularization term and split the joint formulation into several subproblems. We also derive closed-form solutions for each of the subproblems, and thus our ADMM solver can be implemented efficiently. 

Third, theoretical guarantees are established for the structured fusion graph estimator. We provide a non-asymptotic estimation error bound under the high-dimensional setting, which illustrates the impact of several key factors, such as the sample size, the signal dimension, the graph structure, and the topological pattern, on the graph estimation accuracy. Such a statistical analysis provides theoretical evidence that the proposed structured fusion regularization in the joint formulation is able to take advantage of additional sources of information from related graphs to improve statistical efficiency. 

Finally, we apply the proposed graph estimator on both synthetic and real datasets. Experimental results show that our graph estimator can not only find meaningful network topologies from heterogeneous graph signals, but also discover different types of topology patterns in different applications.

The outline of the paper is as follows. We first describe some related works in Section \ref{sec:relatedwork}, and then provide the preliminaries and problem formulation in Section \ref{sec:Preliminaries}. The general graph estimator is introduced in Section \ref{sec:framework}, where we elaborate on the generality of structured fusion regularization and the derivation of the ADMM-based optimization algorithm. Theoretical analysis and experimental results are presented in Section \ref{sec: theorectical} and \ref{sec:experiment}, respectively. We conclude the paper in Section \ref{sec:conclustion}.

\section{Related Work} \label{sec:relatedwork}
In this section, we summarize the differences between our
work and some representative related works on jointly learning multiple graphs.

Most previous approaches for multiple graph learning are based on statistical models. For example, several methods are based on the Gaussian graphical models (GGMs), which jointly learn the precision matrices for the Gaussian-Markov random fields (GMRFs) \cite{danaher2014joint,huang2015joint,hallac2017network, shan2018joint,9277914}. Note that these works focus on learning the precision matrices without assuming the Laplacian constraints. As a result, the learned graphs have both positive and negative edge weights, which may be not suitable for interpreting the structures of data in some contexts, such as the correlations between asset returns in the financial market \cite{10.1093/jjfinec/nbaa018,cardoso2020learning}. In contrast, our proposed method aims to learn the combinatorial graph Laplacian matrices, which can be used to perform harmonic analysis of graph signals and are often desirable in many applications, especially in the field of graph signal processing (GSP) \cite{shuman2013gsp,Dong2016LearningLM}. 

In recent years, there has been clearly a growing interest in learning multiple graphs from a GSP perspective. Generally, the GSP-based graph learning methods can be cast into two categories. One family of approaches focuses on the graph signal properties, which usually assume the signals are stationary or smooth in the sought graphs\cite{8335493,navarro2020joint,saboksayr2021online}.  Another family of approaches focuses on the graph topology properties, such as the most widely studied time-varying graphs\cite{7953415Kalofolias,baingana2016tracking, 2020Time,cardoso2020learning}. Our formulation is a generalization of some of the existing GSP-based graph learning methods. On one hand, in our formulation, we adopt the graph filtered signal models (functions of graph Laplacians), which cater to various classes of graph signals with different frequency properties\cite{egilmez2018graph}. On the other hand, in addition to the time-varying graphs, our proposed structured fusion regularization can flexibly model different types of topological patterns through different Gram matrix choices, making some existing graph estimators as special cases of ours. Please see Section IV-A for details. Moreover, in most existing works, the focus is on problem formulation and algorithmic design, without providing theoretical guarantees of the estimation accuracy. In this paper, non-asymptotic statistical analysis is provided, which proves the consistency of the graph estimator and enables us to investigate the effect of several key factors on the graph estimation error bound.

With regard to the algorithmic design in graph learning, some researchers have recently developed efficient ADMM-based algorithms to solve the single graph learning formulations \cite{8747497,9414791}. However, in our joint formulation with the structured fusion penalty, there are more parameters, additional coupling, and more complicated topological patterns, thus the single graph learning methods cannot be directly applied to solve our multiple cases. Our paper develops a more unified ADMM-based algorithm so that the same framework can incorporate all the different topological patterns that we study. Besides, it covers the single graph learning algorithm \cite{8747497} as a special case of ours, please see Section \ref{sec:optimization} for detailed discussions.

\textbf{Notations:} Boldface upper-case or lower-case letters represent matrices and column vectors, and standard lower-case or upper-case letters
stand for scalars. $\mathbf{A} \succeq 0$ means $\mathbf{A}$ is a positive semi-definite matrix and $\mathbf{A}\geq 0$ means $\mathbf{A}$ is elementwisely larger than $0$. $\sigma_{\mathrm{max}}(\mathbf{A})$ and $\sigma_{\mathrm{min}}(\mathbf{A})$ denote the maximum and minimum singular value of matrix $\mathbf{A}$.   $\mathbf{A}^{\top}, \mathbf{A}^{-1}, \mathbf{A}^{\dagger}, \operatorname{tr}(\mathbf{A}),$  $\operatorname{det}(\mathbf{A})$, and $|\mathbf{A}|_{+}$ denote
the transpose, inverse, pseudo-inverse, trace, determinant, and pseudo-determinant of $\mathbf{A},$ respectively. The vector consisting of all the diagonal elements of $\mathbf{A}$ is denoted by  $\operatorname{diag}(\mathbf{A})$, the $(i,j)$-th entry of $\mathbf{A}$ is denoted by $\mathrm{A}_{ij}$ or $[\mathbf{A}]_{ij}$. The spectral norm and the Frobenius norm of $\mathbf{A}$ are expressed as $\|\mathbf{A} \|_{2}$ and $\|\mathbf{A} \|_{\mathrm{F}}$, respectively. $\|\cdot\|_{p}$ denotes the $\ell_{p}$-norm of a vector. We define $\|\cdot\|_{1,\mathrm{off}}$ as the $\ell_{1}$-norm applied to the off-diagonal matrix entries. $\mathbf{1}$ stands for the all-one vector, and $\mathbf{I}$ stands for the identity matrix. $\odot$ stands for the Hadamard
product, $[K]$ denotes the integer set $\lbrace 1, \cdots, K \rbrace$, and $\text{card}(\cS)$ denotes the number of elements in the set $\cS$. $I(\cdot)$ denotes as an indicator function. Besides, we have $[x]_{+}=\max(x,0)$, $[x]_{-}=\min(x,0)$, and $\langle \mathbf{a},\mathbf{b}\rangle=\mathbf{a}^{\top}\mathbf{b}$ denoting the inner product of vector $\mathbf{a}$ and $\mathbf{b}$.

\section{Preliminaries and Problem formulation} 
 \label{sec:Preliminaries}
\subsection{Fundamentals of GSP}
The design of a graph learning method is based on a properly defined signal model. Here, we introduce the fundamentals of 
GSP that are used to define signal models. Suppose a weighted and undirected graph is represented by a triple $G=(\mathcal{V},\mathcal{E}, \mathbf{L})$, where $\mathcal{V}$ is a finite set of $p$ nodes, $\mathcal{E} \in \mathcal{V} \times \mathcal{V} $ is the edge set, and $\mathbf{L} \in \mathbb{R}^{p\times p} $ is the combinatorial graph Laplacian matrix. The absolute value of the $(i,j)$-th entry of $\mathbf{L}$ is denoted as $|\mathrm{L}_{ij}|$, representing the edge weight between node $i$ and $j$. Since we assume no self-loops in the graph, it satisfies that $ \mathbf{L} \cdot \mathbf{1}=\mathbf{0}$, thereby $\mathbf{L}$ is a symmetric semi-positive definite matrix. In general, the valid set of graph Laplacian matrices is written as 
\begin{equation}
\mathcal{L}=\left\{\mathbf{L}\in \mathbb{R}^{p\times p} \mid \mathbf{L} \succeq 0, \mathbf{L} \cdot \mathbf{1}=\mathbf{0}, \mathrm{L}_{i j}=\mathrm{L}_{j i} \leq 0, i \neq j \right\}.
\label{eq:laset}
\end{equation}

A graph signal is referred as a function $\mathbf{x}: \mathcal{V} \rightarrow \mathbb{R}^p$ that assigns a real value to each graph node. In this paper, we consider a flexible signal model defined by using general filtering operations of signals on graph. Formally, let the eigendecomposition of a graph Laplacian be $\mathbf{L}=\mathbf{U}\mathbf{\Lambda}\mathbf{U}^T$,
then the graph Fourier transform (GFT) of a signal $\mathbf{x}$ is defined as $\hat{\mathbf{x}}=\mathbf{U}^T\mathbf{x}$. Equivalently, the inverse GFT becomes $\mathbf{x}=\mathbf{U}\hat{\mathbf{x}}$.
Using the notion of GFT, a filtered graph signal can be obtained by the following model \cite{egilmez2018graph}
\begin{equation}
\mathbf{x}=\boldsymbol{\mu}+\mathbf{U} h(\mathbf{\Lambda}) \mathbf{U}^T \mathbf{x}_0=\boldsymbol{\mu}+h(\mathbf{L}) \mathbf{x}_0,\label{eq:smod}
\end{equation}
where $h(\cdot)$ is a class of functions, representing graph-based filters, $\mathbf{x}_0 \in \mathbb{R}^p$ is the input data, and $\boldsymbol{\mu}$ is a constant vector. The above filtered signal model (\ref{eq:smod}) provides an unified representation of data residing on  graphs, in which the graph Laplacian matrix captures pairwise relations between the entries of vectorized data and the specific graph-based filter $h(\mathbf{L})$ measures its frequency characteristics.

\subsection{Graph learning from high-dimensional heterogeneous graph signals}
In this paper, we consider graph signals that are observed from $K$ related, but distinct networks $\cG=\{G_1,\cdots,G_K\}$. Each network is an undirected weighted graph $G_k=(\mathcal{V}, \mathcal{E}_k, \mathbf{L}_k)$ with a set of $p$ nodes and a specific set of edges $ \mathcal{E}_k$ characterized by the Laplacian matrix $\mathbf{L}_k \in \mathcal{L}$.
Given that a graph signal $\mathbf{x} \in \mathbb{R}^p$ is collected from a designated graph $G_k$ and the initial data is a white Gaussian signal $\mathbf{x}_0 \sim \mathcal{N}(\boldsymbol{0},\mathbf{I})$, the observed data is thus a sample from a $p$-variate Gaussian distribution $ \mathcal{N} \Big(\boldsymbol{\mu}_k, h^2(\mathbf{L}_k)\Big)$.
 We adopt a decaying function\footnote{ Many other types of $h(\mathbf{L})$ can be chosen for modelling different characteristics of graph signals, we refer to \cite{egilmez2018graph,kalofolias2016learn} for a detailed discussion of graph-based filter functions.} with form $h(\mathbf{L}_k)=\sqrt{\mathbf{L}_k^{\dagger}}$, which represents a class of smooth graph signals and was also studied in previous works\cite{Dong2016LearningLM,kalofolias2016learn,7953415,7979524}. Under these assumptions, the marginal distribution of $\mathbf{x}$ is given by
\begin{equation}
p_k(\mathbf{x} ) = \mathcal{N}(\boldsymbol{\mu}_k,\mathbf{L}_k^{\dagger}).
 \label{eq:smooth}
\end{equation}
It can be seen from ($\ref{eq:smooth}$) that signals collected from $G_k$ follow a degenerate Gaussian distribution with precision matrix $\mathbf{L}_k$.

Suppose that we have total $n=\sum_{k=1}^Kn_k$ sequence of observations $\mathbf{X}=\{\mathbf{x}_i^{(k)}\}_{k=1}^K$, where $\mathbf{x}_i^{(k)}\in \mathbb{R}^p, i=1,\ldots,n_k$ are independent and identically distributed (i.i.d) samples from  (\ref{eq:smooth}). The goal of graph learning is to learn the Laplacian matrix of each graph, i.e.
$\mathbf{L}_1,\ldots,\mathbf{L}_K$, from the observed data $\mathbf{X}$, which is formulated as maximizing the empirical log-likelihood defined by 
\begin{equation} 
\begin{aligned}
\mathcal{F}_n(\mathbf{L}\mid \mathbf{X})
=&\frac{1}{n} \sum_{k=1}^K\sum_{i=1}^{n_k} \log p_k(\mathbf{x}_i^{(k)}) \\
=&\frac{1}{n} \sum_{k=1}^K n_k \Big[ \log(|\mathbf{L}_k|_+)-\tr(\boldsymbol{\Sigma}_k\mathbf{L}_k)\Big].
\end{aligned}
\end{equation}
Here, $\boldsymbol{\Sigma}_k=\sum_{i=1}^{n_k}(\mathbf{x}_i^{(k)} -\boldsymbol{\mu}_k)(\mathbf{x}_i^{(k)}-\boldsymbol{\mu}_k)^{\top}$ denotes the sample covariance of signals associated with graph $G_k$, and from now on,
$
\mathbf{L}=\{\{\mathbf{L}_k\}_{k=1}^K \mid \mathbf{L}_k \in \mathcal{L}, k\in [K] \}
$
represents the set of $K$ graph Laplacians.

In the high-dimensional regime $Kp^2/2 \gg n$, it is well known that the maximum likelihood estimator is not consistent unless additional constraints are imposed on the model. Besides, our goal is to learn graphs that can exhibit both heterogeneity and homogeneity components. Thus, we generalize the learning problem as that of solving the following optimization problem:
\begin{maxi}  
      {\{\mathbf{L}_k\}_{k=1}^{K}}{\mathcal{F}_n(\mathbf{L} \mid \mathbf{X})-\rho_n \mathcal{P}(\mathbf{L})}{}{}
      \addConstraint{\mathbf{L}_k \in \mathcal{L}, ~k \in [K] }.
     \label{eq:updateL}
    \end{maxi} 
Here, the penalty $\mathcal{P}(\mathbf{L})$ is to be specified so as to characterize the topological patterns among multiple graphs and $\rho_n$ is the regularization parameter. In the following, we propose a structured fusion penalty that enables us to incorporate prior knowledge about network topologies in the learning process and will lead to a more general graph learning formulation.

\section{ joint network topology inference via structured fusion regularization} \label{sec:framework}    
\subsection{The proposed structured fusion penalty} \label{sec:proposedestimator}
To facilitate encoding the topological patterns in the learned graphs, we propose a novel structured fusion penalty, which is defined as
\begin{equation}
\begin{aligned}
\mathcal{P}(\mathbf{L})
&=\mathcal{P}_1(\mathbf{L})+\rho \mathcal{P}_2(\mathbf{L})\\
&= \sum_{i \neq j} \|\mathbf{L}_{ij}\|_1+\rho\sum_{i \neq j}\sqrt{ \mathbf{L}^{\top}_{ij} \widetilde{\mathbf{J}} \mathbf{L}_{ij}}.
\end{aligned}
\label{eq:pengl}
\end{equation}
Here, $\mathbf{L}_{ij}=([\mathbf{L}_1]_{ij},\cdots,[\mathbf{L}_K]_{ij})^{\top} \in \mathbb{R}^K, i,j\in [p]$ is a vector of $(i,j)$-entries across the $K$ graph Laplacians, $\widetilde{\mathbf{J}}=\mathbf{J}^{\top}\mathbf{J} \in \mathbb{R}^{K\times K}$ is a Gram matrix with $\mathbf{J}$ being a given matrix or estimated from prior information, and $\rho$ is a non-negative tunning parameter. 

The proposed structured fusion penalty combines a $\ell_1$-norm with a weighted $\ell_2$-norm. The $\ell_1$ norm $\mathcal{P}_1(\mathbf{L})=\sum_{i \neq j} \|\mathbf{L}_{ij}\|_1=\sum_{k=1}^K\|\mathbf{L}_k\|_{1,\text{off}}$ encourages sparsity in each estimated graph Laplacian. We do not penalize the diagonal elements of each $\mathbf{L}_k$, since $\mathbf{L}_k$ is required to satisfy $\mathbf{L}_k\mathbf{1}=\mathbf{0}$.

The weighted $\ell_2$ norm $\mathcal{P}_2(\mathbf{L})$ represents a family of the group-structured norms which has the natural tendency of fusing each group of coefficients according to their correlations\cite{negahban2009unified}. Specifically, the Gram matrix $\widetilde{\mathbf{J}}=\mathbf{J}^{\top}\mathbf{J}$ reflects some underlying geometry or structure across $K$ graphs, and different choices of $\widetilde{\mathbf{J}}$ allow us to enforce different types of topological patterns among graphs. Hence, if we have some prior knowledge of the network topologies, we are able to encode it into $\widetilde{\mathbf{J}}$. In the following, we specify three choices of $\widetilde{\mathbf{J}}$ as instances\footnote{The choice of $\widetilde{\mathbf{J}}$ depends on the desired topological patterns required by the context, there is a wide variety of interesting applications, and what we present in this paper is not meant to be an exhaustive list but rather a small set of illustrative examples that motivated our work.}. 
\begin{itemize}
  \item \textit{The group graph lasso}. Assuming that $\mathbf{J}=\mathbf{I}$ is an identity matrix, then the penalty takes the form
  \begin{equation} \label{eq:golasso}
  \mathcal{P}_2(\mathbf{L})=\sum_{i \neq j} \sqrt{\sum_{k=1}^K [\mathbf{L}_k]_{ij}^2}.
  \end{equation}
This group lasso penalty encourages the Laplacian matrices of $K$ graphs  to be equally similar to each other.  As a result, this penalty is best used in the cases where we expect only a handful of edges varies among the graphs.
  \item \textit{Time-varying graph lasso}. Given that $\mathbf{J}=\left(\begin{array}{ccccc}
0 & 0 & 0 & \ldots & 0 \\
-\omega_{2,1} & 1 & \ddots & \ddots & \vdots \\
0 & -\omega_{3,2} & 1 & \ddots & 0 \\
\vdots & \ddots & \ddots & \ddots & 0 \\
0 & \ldots & 0 & -\omega_{K,K-1} & 1
\end{array}\right)$ is a weighted difference operator, then the penalty has the form
  \begin{equation} \label{eq:timelasso}
  \mathcal{P}_2(\mathbf{L})= \sum_{i \neq j} \sqrt{\sum_{k=2}^{K} \left([\mathbf{L}_k]_{ij}-\omega_{k,k-1}[\mathbf{L}_{k-1}]_{ij}\right)^2}.
  \end{equation}
The weighted parameter $\omega_{k,k-1}>0, \forall k \in [K]$ measures the temporal changes between the neighboring graphs $G_k$ and $G_{k-1}$. If two adjacent graphs change smoothly over time, then the corresponding $\omega_{k,k-1}=1$;  If there is a sudden shift in the graph topology, then $\omega_{k,k-1}$ will be far away from $1$. Thus, the time-varying graph lasso can be applied to encode dynamic evolutionary patterns in the learned graphs.
  \item \textit{The Laplacian shrinkage penalty}. For any $k,k^{\prime}\in [K]$, let $\omega_{k,k^{\prime}} \geq 0 $ measure the pairwise similarity between graphs $G_k$ and $G_{k^{\prime}}$. If $w_{k,k'} >0$, then the graphs $G_k$ and $G_{k^{\prime}}$ are related; If $\omega_{k,k'} =0$, then the graphs $G_k$ and $G_{k^{\prime}}$ are different enough to be considered independent. Setting the entries of the Gram matrix $\widetilde{\mathbf{J}}$ with value $[\widetilde{\mathbf{J}}]_{k k^{\prime}}=\left\{\begin{array}{ll}
\sum_{k^{\prime} \neq k} \omega_{k, k^{\prime}}, & k=k^{\prime} \\
-\omega_{k ,k^{\prime}}, & k \neq k^{\prime} \\
\end{array}\right.$, then the Laplacian shrinkage penalty is defined as
\begin{equation} \label{eq:laplasso}
\mathcal{P}_2(\mathbf{L})= \sum_{i \neq j} \sqrt{\sum_{k,k^{\prime}=1}^{K} \omega_{k,k^{\prime}}([\mathbf{L}_k]_{ij}-[\mathbf{L}_{k^{\prime}}]_{ij})^2}. 
\end{equation}
When two graphs $G_k$ and $G_{k'}$ share more similar components, then the corresponding tuning parameter $\omega_{k,k'}$ will be larger, which encourages many elements of $\mathbf{L}_k, \mathbf{L}_{k'}$ to be identical. This penalty is best used in situations when some graphs are expected to be more similar to each other than others. 

\textbf{Remark 1}: In particular, when $\omega_{k,k'} \rightarrow \infty$, the problem (\ref{eq:updateL}) with the penalty function stated in (\ref{eq:laplasso}) reduces to a single graph learning problem. On the other hand, when $\omega_{k,k'} \rightarrow 0$, the problem  (\ref{eq:updateL}) with the penalty function stated in (\ref{eq:laplasso}) reduces to learning $K$ graphs separately.
\end{itemize}

The above examples highlight the generality of the proposed weighted $\ell_2$ norm, which allows for incorporating various types of prior information of data into the learning process. Namely, any specific topological patterns among graphs, such as a desired level of group sparsity, or a group of time-varying edges, can easily be encoded in the Gram matrix, rendering our graph estimator the ability to model different types of topological patterns flexibly. 

The combined effect of the $\ell_1$ and the weighted $\ell_2$ norm in the proposed penalty is to find estimates of multiple graphs comprising simultaneously of homogeneity and heterogeneity components, which may be desirable in many settings. The simulation
results in Section \ref{sec:experiment} show that the proposed estimator can result in meaningful network topologies.

\subsection{The proposed optimization algorithm} \label{sec:optimization}
With the penalty function stated in (\ref{eq:pengl}), we can obtain a generic formulation for jointly learning multiple graph Laplacians, which is given by
\begin{equation}
\begin{aligned}
\underset{\{\mathbf{L}_k\}_{k=1}^{K}}{ \text{minimize} }~~~ &  \frac{1}{n} \sum_{k=1}^{K} n_k \Big[- \log(|\mathbf{L}_k|_+)+\tr(\boldsymbol{\Sigma}_k\mathbf{L}_k)\Big]
 \\
&+\rho_n \sum_{k=1}^{K} \|\mathbf{L}_{k}\|_{1,\text{off}} +\rho_n\rho\sum_{i \neq j}\|\mathbf{J}\mathbf{L}_{ij}\|_2,  \\
&\begin{array}{l}
\text { s.t. } 
\mathbf{L}_k \in \mathcal{L},k \in [K].
\end{array}
\end{aligned}
\label{eq:nfye}
\end{equation}

 As the regularization term $\sum_{i \neq j}\|\mathbf{J}\mathbf{L}_{ij}\|_2$ makes the parameters in $
\mathbf{L}=\{\mathbf{L}_k\}_{k=1}^K
$ coupled together, it is intractable to solve the whole problem at once. Hence, we propose an algorithm called JEMGL (Jointly Estimation of Multiple Graph Laplacians), which is based on the ADMM method \cite{2010Distributed}, to solve the problem (\ref{eq:nfye}) efficiently. To derive the ADMM solver for the problem (\ref{eq:nfye}), we first decouple the fused parameters and split the problem into a series of sub-problems. We then derive closed-form solutions for each of the sub-problems, and finally, we use a message-passing algorithm to converge on the globally optimal solution.  In the following, we elaborate on the derivation of the ADMM solver.

\textbf{ADMM Solver}:
Setting $\mathbf{M}=\frac{1}{p}\mathbf{1}\mathbf{1}^{\top}$, in \cite{hassan2016topology}, it has been proved that
\begin{equation}
\log \operatorname{det}\left(\mathbf{L}_k+\frac{1}{p} \mathbf{1} \mathbf{1}^{\top}\right)=\log \left(1 \times \prod_{i=2}^{N} \lambda_{i}\right)=\log \left(|\mathbf{L}_k|_{+}\right), 
\end{equation}
so $\log(|\mathbf{L}_k|_+)= \log \det (\mathbf{L}_k + \mathbf{M})$. Let $\mathbf{H}_k=\frac{n \rho_n}{n_k}(\mathbf{I}-\mathbf{1}\mathbf{1}^{\top})$, then $\frac{n \rho_n}{n_k}  \|\mathbf{L}_{k}\|_{1,\text{off}}=\tr(\mathbf{H}_k\mathbf{L}_k)$, and let $\mathbf{Q}_k=\mathbf{\Sigma}_k+\mathbf{H}_k$. Hence, the problem (\ref{eq:nfye}) is equivalent to 
\begin{equation}
\begin{aligned}
 \underset{\{\mathbf{L}_k\}_{k=1}^{K}}{ \text{minimize} }~~~ & \sum_{k=1}^{K} \frac{n_k}{n}  \Big[ -\log \det (\mathbf{L}_k+\mathbf{M})+\tr(\mathbf{Q}_k\mathbf{L}_k)\Big]\\
 &+\rho_n\rho\sum_{i \neq j}\|\mathbf{J}\mathbf{L}_{ij}\|_2,  \\
&\begin{array}{l}
\text { s.t. } 
\mathbf{L}_k \in \mathcal{L}, k \in [K].
\end{array}
\end{aligned}
\label{eq:nfye1}
\end{equation}

 In order to decouple the fused matrices in $\mathbf{L}$ , we introduce two consensus variable sets $\mathbf{A}=\lbrace \mathbf{A}_k  \in \mathbb{R}^{p \times p} \rbrace _{k=1}^K, \mathbf{B}=\lbrace \mathbf{B}_k  \in \mathbb{R}^{p \times p} \rbrace _{k=1}^K$.  Define $\mathbf{a}_{ij}\equiv ([\mathbf{A}_1]_{ij},\cdots,[\mathbf{A}_K]_{ij})^{\top} \in \mathbb{R}^{K}$, $\mathbf{b}_{ij}\equiv ([\mathbf{B}_1]_{ij},\cdots,[\mathbf{B}_K]_{ij})^{\top} \in \mathbb{R}^{K}, i,j=1,\cdots,p$. To dealt with the constraints set $\mathcal{L}$, we adopt the method developed in \cite{8747497}.
Let $\mathbf{C}=\mathbf{1}\mathbf{1}^{\top}-\mathbf{I}$, and define a matrix set
\begin{equation}
\mathcal{B}=\{\tilde{\mathbf{B}} \in \mathbb{R}^{p\times p} \mid \mathbf{I} \odot \tilde{\mathbf{B}} \geq 0,\mathbf{C} \odot \tilde{\mathbf{B}} \leq 0 \}.
\end{equation}
Then, the constraint set $\mathcal{L}$ stated in (\ref{eq:laset}) can be compactly rewritten as
\begin{equation}
\mathcal{L}=\{\mathbf{S} \in \mathbb{R}^{p\times p} \mid \mathbf{S} =\mathbf{P}\bar{\mathbf{\Xi}}\mathbf{P}^{\top}, \bar{\mathbf{\Xi}} \succeq \mathbf{0}, \mathbf{S} \in \mathcal{B} \},
\end{equation}
where $\mathbf{P} \in \mathbb{R}^{p\times (p-1)}$ is the orthogonal complement of the vector $\mathbf{1}$. Define a positive semi-definite matrix set $\mathbf{\Xi}=\lbrace \mathbf{\Xi}_k  \in \mathbb{R}^{p \times p} \mid \mathbf{\Xi}_k\succeq \mathbf{0} \rbrace_{k=1}^K$. Thereby, we have $\mathbf{L}_k =\mathbf{P}\mathbf{\Xi}_k\mathbf{P}^{\top}$, and 
\begin{equation}
\begin{aligned}
\log \det (\mathbf{L}_k+\mathbf{M})&=\log \det (\mathbf{\Xi}_k), \\
\tr(\mathbf{Q}_k\mathbf{L}_k)&=\tr(\tilde{\mathbf{Q}}_k\mathbf{\Xi}_k),
\end{aligned}
\end{equation}
with $\tilde{\mathbf{Q}}_k=\mathbf{P}^{\top}\mathbf{Q}_k\mathbf{P}$.
 To facilitate the computation, we assume $\widetilde{\mathbf{J}} =\mathbf{J}^{\top}\mathbf{J}$ is a positive definite matrix. Thus, the problem (\ref{eq:nfye1}) is equivalent to 
\begin{equation}
\begin{aligned}
\underset{\{\mathbf{\Xi},\mathbf{A}, \mathbf{B}\}}{ \text{minimize} } ~~~ & \sum_{k=1}^{K} \frac{n_k}{n}  \Big[- \log \det (\mathbf{\Xi}_k)+\tr(\tilde{\mathbf{Q}}_k\mathbf{\Xi}_k)\Big] \\
 &+\rho_n\rho\sum_{i \neq j}\|\mathbf{J}\mathbf{a}_{ij} \|_2,  \\
&\begin{array}{l}
\text { s.t. } ~~\mathbf{\Xi}_k \succeq \mathbf{0}, \quad k \in [K], \\
~~~~~~~\mathbf{P}\mathbf{\Xi}_k \mathbf{P}^{\top}=\mathbf{B}_k,\quad k \in [K],\\
~~~~~~~\mathbf{J}\mathbf{a}_{ij}=\mathbf{J}\mathbf{b}_{ij},\quad i,j \in [p], \text{and}~i \neq j, \\
~~~~~~~\mathbf{B}_k \in \mathcal{B},\quad k \in [K].\\
\end{array}
\end{aligned}
\label{eq:nfye2}
\end{equation}

Let $\mathbf{E}=\{ \mathbf{E}_{k}\in \mathbb{R}^{p \times p} \}_{k=1}^K$ and $\mathbf{F}=\{ \mathbf{F}_{k}\in \mathbb{R}^{p \times p} \}_{k=1}^K$ denote Lagrange multiplier matrices. Define $\mathbf{e}_{ij}\equiv \{[\mathbf{E}_1]_{ij},\cdots,[\mathbf{E}_K]_{ij} \} \in \mathbb{R}^{K}$ and $\mathbf{f}_{ij}\equiv ([\mathbf{F}_{1}]_{ij},\cdots,[\mathbf{F}_{K}]_{ij})^{\top} \in \mathbb{R}^{K},i,j=1,\cdots,p$. Then the corresponding (partial) augmented
Lagrangian is given by
\begin{equation}
\begin{aligned}
& \mathcal{L}_\rho(\mathbf{\Xi},\mathbf{A},\mathbf{B},\mathbf{E},\mathbf{F})\\
  =& \sum_{k=1}^{K} \frac{n_k}{n}  \Big[- \log \det (\mathbf{\Xi}_k)+\tr(\tilde{\mathbf{Q}}_k\mathbf{\Xi}_k)\Big] \\
 &+\rho_n\rho\sum_{i \neq j}\|\mathbf{J}\mathbf{a}_{ij} \|_2  \\
 & +\sum_{k=1}^{K}\Big [\tr\Big(\mathbf{E}_k^{\top}(\mathbf{P}\mathbf{\Xi}_k \mathbf{P}^{\top}-\mathbf{B}_k)\Big) +\frac{\rho}{2}  \| \mathbf{P}\mathbf{\Xi}_k \mathbf{P}^{\top}-\mathbf{B}_k\|_{\mathrm{F}}^2\Big]\\
 &+\sum_{i \neq j} \Big[\mathbf{f}_{ij}^{\top}(\mathbf{J}\mathbf{a}_{ij}-\mathbf{J}\mathbf{b}_{ij})+\frac{\rho}{2} \|\mathbf{J}\mathbf{a}_{ij}-\mathbf{J}\mathbf{b}_{ij}\|_2^2\Big],\\
\end{aligned}
\label{eq:scaleLagrang}
\end{equation}
where $\rho >0$ is the ADMM penalty parameter.
ADMM consists of the following
updates, where $m$ denotes the iteration number:
\begin{equation}
\left\{\begin{array}{l}
\mathbf{\Xi}^{(m+1)}:=\underset{\mathbf{\Xi} \succeq 0}{\arg \min}  \mathcal{L}_\rho(\mathbf{\Xi},\mathbf{A}^{(m)},\mathbf{B}^{(m)},\mathbf{E}^{(m)},\mathbf{F}^{(m)}),\\
\mathbf{A}^{(m+1)}:=\underset{\mathbf{A}} {\arg \min } \mathcal{L}_\rho(\mathbf{\Xi}^{(m+1)},\mathbf{A},\mathbf{B}^{(m)},\mathbf{E}^{(m)},\mathbf{F}^{(m)}),\\
\mathbf{B}^{(m+1)}:=\underset{\mathbf{B}}{\arg \min}   \mathcal{L}_\rho(\mathbf{\Xi}^{(m+1)},\mathbf{A}^{(m+1)},\mathbf{B},\mathbf{E}^{(m)},\mathbf{F}^{(m)}).
\end{array}\right.
\end{equation}

\subsubsection{Update of $\mathbf{\Xi}$}
The $\mathbf{\Xi}$-step can be split into separate updates for each $\mathbf{\Xi}_k$ , which can then be solved in parallel:
\begin{equation}
\begin{aligned}
\underset{\mathbf{\Xi}_{k} \succeq 0}{ \text{minimize} }  &-\log \det (\mathbf{\Xi}_k) +\tr(\tilde{\mathbf{Q}}_k\mathbf{\Xi}_k)+\tr(\mathbf{P}^{\top}\mathbf{E}_k^{(m)}\mathbf{P}\mathbf{\Xi}_k) \\ 
&+\frac{\rho}{2}  \| \mathbf{P}\mathbf{\Xi}_k \mathbf{P}^{\top}-\mathbf{B}_k^{(m)}\|_{\mathrm{F}}^2.  \\ 
\end{aligned}
\label{eq:ADMM-L}
\end{equation} 
The problem (\ref{eq:ADMM-L}) can be solved by applying the method developed in pages 46-47 of the book\cite{2010Distributed}, the solution is 
\begin{equation} \label{eq:updatexi}
\mathbf{\Xi}_k^{(m+1)}=\mathbf{V}_k\mathbf{D}_k\mathbf{V}_k^{\top}
\end{equation}
with $\frac{1}{\rho} \mathbf{P}^{\top}(\mathbf{Q}_k+\mathbf{E}_k^{(m)}-\rho \mathbf{B}_k^{(m)}) \mathbf{P}=\mathbf{V}_k \mathbf{\Lambda}_k \mathbf{V}^{\top}_k,$ and $\mathbf{D}_k$ is a diagonal
matrix with $[\mathbf{D}_k]_{i i}=\frac{-\rho \Lambda_{i i}+\sqrt{\rho^{2} \Lambda_{i i}^{2}+4 \rho}}{2 \rho}$
.

\subsubsection{Update of $\mathbf{A}$}
The update of $\mathbf{A}$ is to compute
\begin{equation} \label{eq:updateaij}
\begin{aligned}
\mathbf{a}_{i j}^{(m+1)}=\left[1-\frac{\rho_{n} \rho_{1}}{\rho\left\|\mathbf{J} \mathbf{b}_{i j}^{(m)}+\mathbf{f}_{ij}^{(m)}\right\|_2}\right]_{+}\left(\mathbf{J} \mathbf{b}_{i j}^{(m)}+\mathbf{f}_{ij}^{(m)}\right).
\end{aligned}
\end{equation}

\begin{algorithm}[!t]
\setlength{\belowcaptionskip}{-0.6cm}
\caption{\textbf{:} ADMM-based Algorithm for Joint Estimation of Multiple Graph Laplacians   (JEMGL)  }
\label{algo: ADMMGL}
\hspace*{0.02in} {\bf Input:} 
$\mathbf{P}, \mathbf{J}, \{\mathbf{Q}_k \}_{k=1}^K$, symmetric $\mathbf{A}^{(0)}$,$\mathbf{B}^{(0)}$,$\mathbf{E}^{(0)}$ and $\mathbf{F}^{(0)}$, $\rho_n>0,\rho>0$, m=0; \\ 
\hspace*{0.02in} {\bf Output:} 
 $\hat{\mathbf{L}}$\\
\hspace*{0.02in} {\bf Repeat:}
\begin{algorithmic}[1] 
\FOR {\textit{k}=1,\ldots, \textit{K} (in parallel)}
\STATE{ Update $ \mathbf{\Xi}_k^{(m+1)}$ according to (\ref{eq:updatexi})};\\
\STATE{Update $  \mathbf{L}_k^{(m+1)}=\mathbf{P} \mathbf{\Xi}_k^{(m+1)} \mathbf{P}^{\top}$};\\
\ENDFOR
\STATE{Update $\mathbf{a}_{i j}^{(m+1)}$ according to (\ref{eq:updateaij})};\\
\STATE{When $i=j$, update $\mathbf{b}_{i j}^{(m+1)}$ according to (\ref{eq:updatebii}) };\\
\STATE{ When $i\neq j$, update $\mathbf{b}_{i j}^{(m+1)}$ according to (\ref{eq:updatebij}) ;
}\\
\STATE{Update $\mathbf{E}_k^{(m+1)}, \mathbf{f}_{ij}^{(m+1)}$ according to (\ref{eq:updateef})};\\
\STATE{$m=m+1$};\\
\STATE{\textbf{Until} convergence}
\RETURN { $\hat{\mathbf{L}}= \{\mathbf{L}_k^{(m)}\}_{k=1}^K$}.
\end{algorithmic}
\end{algorithm}

\subsubsection{Update of $\mathbf{B}$}
Updating $\mathbf{B}$ is equivalent to solving the following optimization problem:
\begin{equation}
\begin{aligned}
 &\underset{\{\mathbf{B}_k \in \mathcal{B}\}_{k=1}^K}{\text{minimize}} \sum_{k=1}^K-\Big[\tr({\mathbf{E}_k^{(m)}}^{\top}\mathbf{B}_k) +\frac{\rho}{2}  \| \mathbf{L}_k^{(m+1)} -\mathbf{B}_k\|_F^2\Big] \\
&+\sum_{i \neq j} \Big[{\mathbf{f}_{ij}^{(m)}}^{\top}(\mathbf{J}\mathbf{a}_{ij}^{(m+1)}-\mathbf{J}\mathbf{b}_{ij})+\frac{\rho}{2} \|\mathbf{J}\mathbf{a}_{ij}^{(m+1)}-\mathbf{J}\mathbf{b}_{ij}\|_2^2\Big],
\end{aligned}
\label{eq:ADMM-B}
\end{equation} 
where $\mathbf{L}_k^{(m+1)}=\mathbf{P}\mathbf{\Xi}_k^{(m+1)}\mathbf{P}^{\top}$.
Let $\mathbf{l}^{(m+1)}_{ij}=\{[\mathbf{L}_1^{(m+1)}]_{ij},\cdots,[\mathbf{L}_K^{(m+1)}]_{ij} \} \in \mathbb{R}^{K}$.
When $i=j$, then the solution to the problem (\ref{eq:ADMM-B} ) is 
\begin{equation} \label{eq:updatebii}
\begin{aligned}
\mathbf{b}_{ij}^{(m+1)}:&=
\underset{\{\mathbf{B}_k \in \mathcal{B}\}_{k=1}^K}{\arg \min} -{\mathbf{e}^{(m)}_{ij}}^{\top}\mathbf{b}_{ij}+\frac{\rho}{2}\|\mathbf{l}^{(m+1)}_{ij}-\mathbf{b}_{ij}\|_2^2,\\
&=\left[ \frac{1}{\rho}\mathbf{e}^{(m)}_{ij}+\mathbf{l}^{(m+1)}_{ij}\right]_+;
\end{aligned}
\end{equation}
When $i \neq j$, then the solution is
\begin{equation} 
\begin{aligned}
\mathbf{b}_{ij}^{(m+1)}:=\underset{\{\mathbf{B}_k  \in \mathcal{B}\}_{k=1}^K}{\arg \min} & -{\mathbf{e}^{(m)}_{ij}}^{\top}\mathbf{b}_{ij}+\frac{\rho}{2}\|\mathbf{l}^{(m+1)}_{ij}-\mathbf{b}_{ij}\|_2^2\\
&- {\mathbf{f}_{ij}^{(m)}}^{\top}\mathbf{J}\mathbf{b}_{ij}+\frac{\rho}{2}\|\mathbf{J}\mathbf{a}_{ij}^{(m+1)}-\mathbf{J}\mathbf{b}_{ij}\|_2^2. \\
\end{aligned} 
\end{equation}
By some linear algebra, we have
\begin{equation} \label{eq:updatebij}
\begin{aligned}
&\mathbf{b}_{ij}^{(m+1)}\\
=&\left[(\mathbf{I}+\widetilde{\mathbf{J}})^{-1}(\frac{1}{\rho}\mathbf{e}^{(m)}_{ij}+\frac{1}{\rho}\mathbf{J}^{\top}\mathbf{f}_{ij}^{(m)}+\mathbf{l}^{(m+1)}_{ij}+\widetilde{\mathbf{J}} \mathbf{a}_{ij}^{(m+1)})\right]_{-}. \\
\end{aligned} 
\end{equation}
\subsubsection{Update of $\mathbf{E}$ and $\mathbf{F}$}
\begin{equation} \label{eq:updateef}
\begin{aligned}
\mathbf{E}_k^{(m+1)}=\mathbf{E}_k^{(m)}+\rho(\mathbf{L}_k^{(m+1)}-\mathbf{B}_k^{(m+1)}),\\
\mathbf{f}_{ij}^{(m+1)}=\mathbf{f}_{ij}^{(m)}+\rho\mathbf{J}(\mathbf{a}_{ij}^{(m+1)}-\mathbf{b}_{ij}^{(m+1)}).
\end{aligned}
\end{equation}

\textbf{Global Convergence and Computation Complexity}. The ADMM-update procedure of the JEMGL is summarized in Algorithm \ref{algo: ADMMGL}. By separating the problem (\ref{eq:nfye}) into three blocks of
variables, $\mathbf{\Xi}$, $\mathbf{A}$ and $\mathbf{B}$, the JEMGL algorithm is guaranteed to converge to the global optimum. During the implementation the
JEMGL algorithm, we adopt an adaptive
update scheme for $\rho$ as suggested in Sec 3.4.1 of \cite{2010Distributed}, such that $\rho$ varies in every iteration and becomes less dependent on the initial choice. In the Algorithm \ref{algo: ADMMGL}, the most computationally demanding steps are the eigenvalue decomposition (line 2) and the matrix multiplication operations (line 3), whose computational complexities are both $\cO(p^3)$. Therefore, the computational complexity of Algorithm \ref{algo: ADMMGL} is in the order of $\cO(p^3)$ per iteration.

\textbf{Remark 2}: By giving different initialization on the Gram matrix $\mathbf{J}$, the JEMGL algorithm can learn multiple graphs with different topological patterns without any modification. Specifically, by setting $\mathbf{J}=\mathbf{0}$, the JEMGL algorithm learns $K$ graphs separately, and when $K=1$, it learns a single graph. Therefore, the single-graph learning algorithm in \cite{8747497} can be regarded as a special case of ours.

\section{Theoretical Guarantees} \label{sec: theorectical}
In this section, we derive the non-asymptotic estimation error bound of our structured fusion graph estimator. To derive this bound, we generally adopt tools from \cite{negahban2009unified} and \cite{2011penalized}, properly adjusted to our problem. We consider a high-dimensional setting $Kp^2/2 \gg n$, where both $n$ and $p$ go to infinity. 

 Let $\mathbf{L}^*=\{\mathbf{L}_k^* \in \mathcal{L}\}_{k=1}^{K}$ be the set of the true Laplacian matrices of $K$ graphs and $\hat{\mathbf{L}}_{\rho_n}=\{\hat{\mathbf{L}}_k \in \mathcal{L}\}_{k=1}^{K}$ be the optimal solution of the structured fusion graph estimator (\ref{eq:nfye}) with a fixed regularization parameter $\rho_n$. The following theorem establishes bounds and hence convergence rates for the error $\left\|\hat{\mathbf{L}}_{\rho_n}-\mathbf{L}^*\right\|$, in Frobenius norm. 

\begin{theorem} \label{theorem:mainresults}
Let  $s=\#\{(i,j): [\mathbf{L}_{k}^*]_{ij}\neq 0, k \in [K],i,j=1,\cdots,p,i \neq j\}$ denote the sparsity parameter. Suppose that $\tau \in\left(0, \min _{k} \frac{n_k}{n}\right)$. For
$$
n \geq \max \left\{\frac{2\ln p}{\tau}, \frac{2^{13} 15^2\lambda_{\mathbf{L}}^2\kappa_{\widetilde{\mathbf{J}}}^2 \nu^2}{\tau^3} s\ln p\right\}
$$
and 
$\rho_{n}=2 \left(1+ \sigma_{\min}(\widetilde{\mathbf{J}})\sqrt{K}\right)\left(\frac{1}{p}+ 40 \sqrt{2} \nu \sqrt{\frac{\ln p}{n \tau}}\right)$, 
 we have probability at least $(1-2 K / p)$ such that
$$
\left\|\hat{\mathbf{L}}_{\rho_n}-\mathbf{L}^*\right\|_{\mathrm{F}} \leq 24 \kappa_{\widetilde{\mathbf{J}}}\lambda_{\mathbf{L}}^2\tau^{-3/2}\left(\frac{\sqrt{s}}{p}+40\sqrt{2} \nu \sqrt{\frac{s \ln p}{n}}\right),
$$
where $\kappa_{\widetilde{\mathbf{J}}}=\left(1+ \sigma_{\min}(\widetilde{\mathbf{J}})\sqrt{K}\right)\left(1+\rho \sqrt{\sigma _{\max }(\widetilde{\mathbf{J}})}\right)$, $\lambda_{\mathbf{L}}= \max_k \|\mathbf{L}_k^*\|_{\mathrm{2}} $ and $\nu=\max_{k,i}[(\mathbf{L}_k^*)^{\dagger}]_{ii}$.
\end{theorem}

The proof of Theorem $\ref{theorem:mainresults}$ can be found in Appendix \ref{sec:prooftheorem}.

Theorem $\ref{theorem:mainresults}$ characterizes the estimation error bound with respect to some factors, such as the number of data samples $n$, the signal dimension $p$, the number of
multiple related networks $K$, and some conditions on the true graph Laplacians. Besides, it also provides insights into the impact of the spectral properties of $\widetilde{\mathbf{J}}$ on the estimation accuracy. For instance, if $K$ graphs are highly similar, then $\sigma_{\min}(\widetilde{\mathbf{J}}) \approx 0$, leading to smaller sample complexity and estimation error bound. This makes sense, as information
can be better shared when estimating parameters of similar graphs. These theoretical results suggest that our proposed structured fusion penalty is able to take advantage of the network topology information to improve statistical estimation efficiency.

\section{Experiments} \label{sec:experiment}
In this section, we evaluate our proposed framework for joint inference of networks on both synthetic and real data. We first introduce the general experimental settings and then compare the performance of our methods with those of benchmark methods.

\subsection{Experiment setting}
\subsubsection{Baselines for comparison}
We compare the performance of our proposed methods with two state-of-the-art methods: time-varying graph learning with Tikhonov regularization (TGL-Tikhonov) \cite{7953415Kalofolias} and the graphical Lasso based model for combinatorial graph Laplacian estimation (CGL-GLasso) \cite{7979524}. For CGL-GLasso,  estimation
was carried out separately for each class with the same regularization parameter (see Table \ref{tab:Comparisionofmethod} for a summary of all the methods we compare). We apply our JEMGL algorithm with different choices of Gram matrix and demonstrate the importance of using appropriate penalties for different types of topological patterns among graphs.
\begin{table*}[h]
\renewcommand{\arraystretch}{1.2} 
\caption{List of alternative and proposed methods}\label{tab:Comparisionofmethod} 
\centering

\begin{tabular}{|c|c|c|}  
		\hline\hline  

		Method & Algorithm & Penalty\\  
		\hline
				CGL-GLasso & Single combinatorial graph Laplacian learning from smooth graph signals \cite{7979524}& Graphical Lasso \\
		
		\hline
		TGL-Tikhonov &  Time-varying graph learning from smooth graph signals \cite{7953415Kalofolias} & Tikhonov regularization \\
		\hline

		JEMGL-GGL (proposed) & Joint estimation of multiple graph Laplacians & Group graph Lasso (\ref{eq:golasso}) \\
		\hline
		JEMGL-TVGL (proposed) &Joint estimation of multiple graph Laplacians &Time-varing graph Lasso (\ref{eq:timelasso}) \\
		\hline
		JEMGL-LSP (proposed) & Joint estimation of multiple graph Laplacians & Laplacian shrinkage penalty (\ref{eq:laplasso}) \\
		\hline
	\end{tabular}
\setlength{\belowcaptionskip}{-0.5cm}
\end{table*}
\subsubsection{Evaluation metrics}
To measure the graph Laplacians estimation quality, we calculate the relative error (RE) and F-score (FS) used in \cite{7979524}, each average over all $K$ related graphs. RE is given by
\begin{equation}
\mathrm{RE}:=\frac{1}{K} \sum_{k=1}^{K} \frac{\left\|\widehat{\mathbf{L}}_k-\mathbf{L}_k^*\right\|_{\mathrm{F}}}{\|\mathbf{L}_k^*\|_{\mathrm{F}}} ,
\end{equation}
where $\widehat{\mathbf{L}}_k$ is the estimated graph Laplacian matrix of graph $G_k$, and $\mathbf{L}_k^*$ is the ground truth. RE reflects the accuracy of edge weights on the estimated graph. The FS is given by
\begin{equation}
\mathrm{FS}:=  \frac{1}{K} \sum_{k=1}^K \frac{2 \mathrm{tp}_k}{2\mathrm{tp}_k+\mathrm{fn}_k+\mathrm{fp}_k} ,
\end{equation}
where the true positive ($\mathrm{tp}_k$) is the total number of edges that are included both in $\widehat{\mathbf{L}}_k$
and $\mathbf{L}_k^*$, the false negative ($\mathrm{fn}_k$) is the number of edges that are not included in $\widehat{\mathbf{L}}_k$ but are included
in $\mathbf{L}_k^*$, and the false positive ($\mathrm{fp}_k$) is the number of edges that are included in $\widehat{\mathbf{L}}_k$
but are not included in $\mathbf{L}_k^*$. The FS measures the accuracy of the estimated graph topology. The higher the
FS is, the better the performance of capturing graph
topology is.
\subsection{Experiments on Synthetic Datasets} \label{sec:differentpatterns}
\subsubsection{Network construction} 
In the simulation, we consider observations from $K=3$ related graphs. In the following, we elaborate on the three different types of topological patterns among $K$ graphs.
\begin{itemize}
\item \textbf{Pattern 1.} The three graphs share a similar topology with only a handful of edges vary among the graphs.  We set the graph adjacency matrix $\mathbf{W}_k=\mathbf{W}_c +\mathbf{U}_k, k\in [K]$, where 
$\mathbf{W}_c$ is common in all graphs and $\mathbf{U}_k$ represents
unique structure of the $k$-th graph. The common part, 
$\mathbf{W}_c$, is generated as follows: we construct an undirected graph $G_c$ of size $p$ following an Erd{\H{o}}s-R{\'e}nyi model \cite{erdHos1960evolution} with an edge connection probability $0.2$ and edge weights uniformly from $[0.75,2]$ . The adjacency matrix of graph $G_c$ represents $\mathbf{W}_c$. For each $\mathbf{U}_k$, we first set $\mathbf{U}_k=\mathbf{0}$,  then we randomly pick $5\%$ pairs of symmetric off-diagonal entries and replace them with values randomly chosen from the interval $[-1,-0.5]\cup [0.5, 1]$. Finally, for each $\mathbf{W}_k$, we normalize the edge weighs to $[0,2]$ by setting $[\mathbf{W}_k]_{i j}=\left\{\begin{array}{cl} 0, & [\mathbf{W}_k]_{i j}<0; \\ 2, & [\mathbf{W}_k]_{i j}>2. \end{array}\right.$

\item \textbf{Pattern 2. } Three time-varying graphs. We first construct an Erd{\H{o}}s-R{\'e}nyi graph $G_1$ with $p$ nodes attached to other nodes with probability 0.3. The edge weights uniformly from $[0.75,2]$. The second graph $G_2$ is obtained by randomly down-sampling $10\%$ of edges from $G_1$. Similarly, we obtain $G_3$ by randomly down-sampling $10\%$ edges from $G_2$.

\item \textbf{Pattern 3. } Three graphs with two graphs are expected to be more similar to each other than the other one. We first generate a random modular graph $G_2$ with $p$ nodes and $3$ modules (subgraphs), where the node attachment probability across modules and within modules are $0.1$ and $0.5$ respectively. We randomly assign $\mathrm{Unif}(0.75,2)$ values to non-zero entries of the corresponding adjacency matrix $\mathbf{W}_2$. For each graph $G_1$ and $G_3$, we remove one of the modules of $G_2$ by setting the corresponding off-diagonal entries of $\mathbf{W}_2$ to $0$.
\end{itemize}

In the above setting, we use the pattern index to represent the level of heterogeneity among the three graphs. As the pattern index increases, we gradually increase the proportion of individual connectivity in each graph. (i.e., graphs with pattern 1 share the largest ratio of common structure, while graphs with pattern 3 share the smallest ratio of common structure.)
\subsubsection{Graph signals generation}
After three graphs are constructed based on a topological pattern, the data are randomly generated through a smooth graph signal model $\mathbf{x}_i^{(k)} \sim \mathcal{N} \Big(\boldsymbol{0}, \mathbf{L}_k^{\dagger}\Big),k \in [K]$. 
\subsubsection{Implementation}
For the JEMGL algorithm, we chose $\rho_n$ by conducting a grid search over tuning range $10^{-2+2r/15}$ with $r=0,1,\ldots,20$. The Gram matrix $\widetilde{\mathbf{J}}$ in time-varying graph lasso is initialized to $\omega_{k,k-1}=1, \forall 1<k \leq K,$ in all settings. The Laplacian shrinkage penalty (\ref{eq:laplasso}) is initialized as follows: For graphs with pattern 1, we set $\omega_{k,k^{\prime}}=1, \forall k, k^{\prime} \in [K]$; For graphs with pattern 2, we set $\omega_{k ,k^{\prime}}=\left\{\begin{array}{ll}
0.5, & |k - k^{\prime}|>1; \\
1, & |k - k^{\prime}|=1. \\
\end{array}\right.$; For graphs with pattern 3, we set $\omega_{k, k^{\prime}}=\left\{\begin{array}{ll}
0.1, & |k - k^{\prime}|>1; \\
1, & |k - k^{\prime}|=1. \\
\end{array}\right.$
For the CGL-GLasso \cite{7979524} and the TGL-Tikhonov \cite{7953415Kalofolias}, the tuning parameter $\alpha$ is selected from the following set: 
\begin{equation*}
\{0\} \cup\left\{0.75^{r}\left(z_{\max } \sqrt{\log p / n}\right) \mid r=1,2,3, \cdots, 14\right\},
\end{equation*}
where $z_{\max}=\max_{i\neq j}[\mathbf{L}_k^{\dagger}]_{ij}$.  All algorithms are terminated when the Frobenius norm of the change of $\mathbf{L}$ between iterations is smaller than a
threshold (by default $10^{-3}$).
\subsubsection{Performance comparison} 
In these experiments, our goal is to compare the best achievable performance of all methods. We set the number of graph nodes $p=15$ and consider two cases: balanced sample sizes $n=(100, 100,100)$, unbalanced sample sizes $n=(60,90,150)$. We performed 50 Monte-Carlo simulations for each set-up and reported the averaged relative error and F-score for each method. Table \ref{tab:ComparisionofPerformance} (a) are comparison summaries for graphs with pattern 1, Table \ref{tab:ComparisionofPerformance} (b) are comparison summaries for graphs with pattern 2 and Table \ref{tab:ComparisionofPerformance} (c) are comparison summaries for graphs with  pattern 3. 
\begin{table}[t]
\renewcommand{\arraystretch}{1.2} 
\setlength{\abovecaptionskip}{2 pt} 
\setlength{\belowcaptionskip}{0 cm}
\caption{The performance of learning multiple graphs with different patterns. (a) Comparison summaries for graphs with pattern 1. (b) Comparison summaries for graphs with pattern 2. (c) Comparison summaries for graphs with pattern 3. }\label{tab:ComparisionofPerformance} 
\centering
\subtable[ Pattern 1]{
\begin{tabular}{ |c|c|c|c|}
\hline\hline
		$n$ & Methods  & RE&FS \\  
		\hline
		\multirow{5}{*}{$(100,100,100)$} &CGL-GLasso  &0.492 &0.473\\ 
& TGL-Tikhonov  &0.319&0.670 \\
& JEMGL-GGL &\textbf{0.213}&\textbf{0.843}\\ 
& JEMGL-TVGL &0.314 &0.678\\
& JEMGL-LSP &0.222 &0.763\\
\hline
\multirow{5}{*}{$(60,90,150)$} &  CGL-GLasso &0.561&0.424\\
& TGL-Tikhonov  &0.431& 0.552 \\
& JEMGL-GGL &\textbf{0.333}&\textbf{0.652}\\
& JEMGL-TVGL&0.437&0.548\\
&  JEMGL-LSP &0.358 &0.641\\
\hline
\end{tabular}
}
\subtable[Pattern 2]{
\begin{tabular}{ |c|c|c|c|}
\hline\hline
		$n$ & Methods  & RE&FS \\  
		\hline
		\multirow{5}{*}{$(100,100,100)$} & CGL-GLasso  &0.412&0.561\\
& TGL-Tikhonov &0.237&0.748 \\
& JEMGL-GGL &0.374&0.625\\
& JEMGL-TVGL &\textbf{0.231}&\textbf{0.753}\\
&  JEMGL-LSP &0.286&0.701\\
\hline
\multirow{5}{*}{$(60,90,150)$} &  CGL-GLasso &0.502&0.451\\
& TGL-Tikhonov  &\textbf{0.362}&\textbf{0.639} \\
& JEMGL-GGL&0.478&0.512\\
& JEMGL-TVGL &0.366&0.631\\
&  JEMGL-LSP &0.387&0.601\\
\hline
\end{tabular}
}
\subtable[ Pattern 3]{
\begin{tabular}{ |c|c|c|c|}
\hline\hline
		$n$ & Methods  & RE&FS \\ 
		\hline
		\multirow{5}{*}{$(100,100,100)$} & CGL-GLasso  &0.401 &0.584\\
& TGL-Tikhonov  &0.391&0.592 \\
& JEMGL-GGL &0.395 &0.589\\
& JEMGL-TVGL &0.389&0.599\\
&  JEMGL-LSP&\textbf{0.301}&\textbf{0.695}\\
\hline
\multirow{5}{*}{$(60,90,150)$} &  CGL-GLasso  &0.492&0.473\\
& TGL-Tikhonov  &0.482&0.507 \\
& JEMGL-GGL &0.487 &0.476\\
& JEMGL-TVGL &0.484&0.493\\
&  JEMGL-LSP &\textbf{0.460}&\textbf{0.537}\\
\hline
\end{tabular}
}
\setlength{\belowcaptionskip}{0 cm}
\end{table}

Overall, as shown in Table \ref{tab:ComparisionofPerformance}, all methods are affected by unbalanced sample size and present worse performance in terms of estimation accuracy. Noted also that the separate estimation method (CGL-GLasso) gives the best results for graphs with pattern 3 and the worst results for graphs with pattern 1, which is exactly opposite to the results of the other four joint estimation methods. This is precisely as it should be, since
the joint estimation methods have the advantage of learning related networks by sharing information. As the pattern index increases, the topologies become more and more different, and the results of the joint and separate methods will move closer. This indicates that as long as a substantial common structure among graphs exists, the joint estimation methods are superior to the separate estimation method.

Additionally, considering the results in each sub-table, we can see that, regardless of penalty type, our JEMGL algorithms outperform the two baselines in most scenarios, except the TGL-Tikhonov methods slightly outperforming our methods on graphs with pattern 2. These empirical results demonstrate that the proposed structured fusion penalty is able to capture various types of topological patterns among graphs with high precision.

\textbf{Choice of Penalty Type}. While the JEMGL outperforms the two baselines regardless of the penalty type, even greater gains can be achieved by choosing a proper penalty. As shown in Table \ref{tab:ComparisionofPerformance}, there are clear benefits from using certain penalties in certain situations. For example, graphs with pattern 1, which is well-suited to be analysed by a group graph lasso, choosing this
penalty leads to an $8\%$ higher F-score. For graphs with pattern 3, the Laplacian shrinkage penalty does the best job at recovering the
three graphs. In real-world cases, the Gram matrix $\widetilde{\mathbf{J}}$ can be constructed based on exactly what type of topological patterns one is looking for in the data,  using the descriptions in Section \ref{sec:proposedestimator}.

\begin{figure}[!t] 
\setlength{\belowcaptionskip}{-3cm}
\centering
\subfigure[RE]{
\includegraphics[width=0.7\columnwidth , height=0.4 \columnwidth]{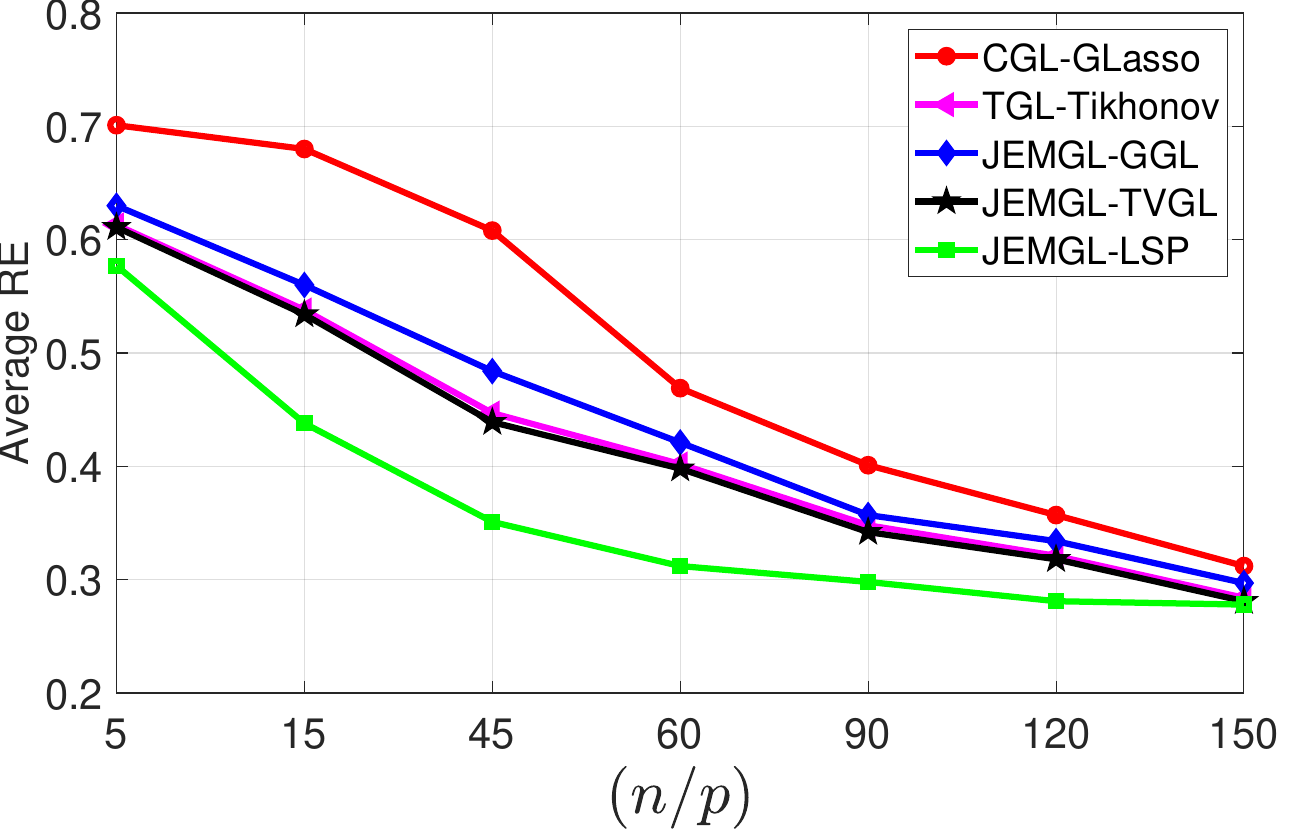}
}
\hspace{0.2ex}
\subfigure[FS]{
\includegraphics[width=0.7\columnwidth, height=0.4 \columnwidth ]{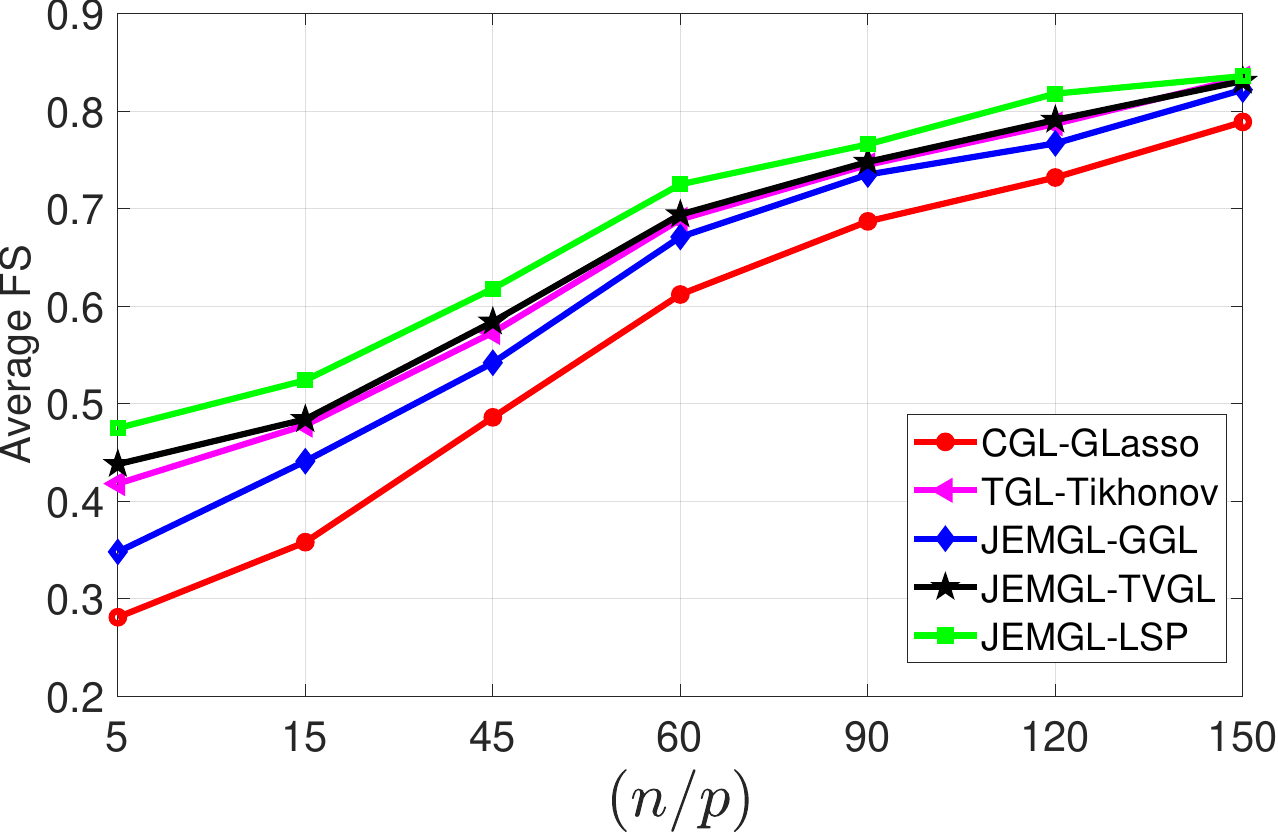}
}

\setlength{\belowcaptionskip}{-0.12 cm}
\caption{The effects of the sample size and the signal dimension $(n/p)$ on the performance of all the methods. }
 \label{fig:simulation1}
\end{figure}

\begin{figure}[!t] 
\setlength{\belowcaptionskip}{-3cm}
\centering
\subfigure[RE]{
\includegraphics[width=0.7\columnwidth , height=0.38 \columnwidth]{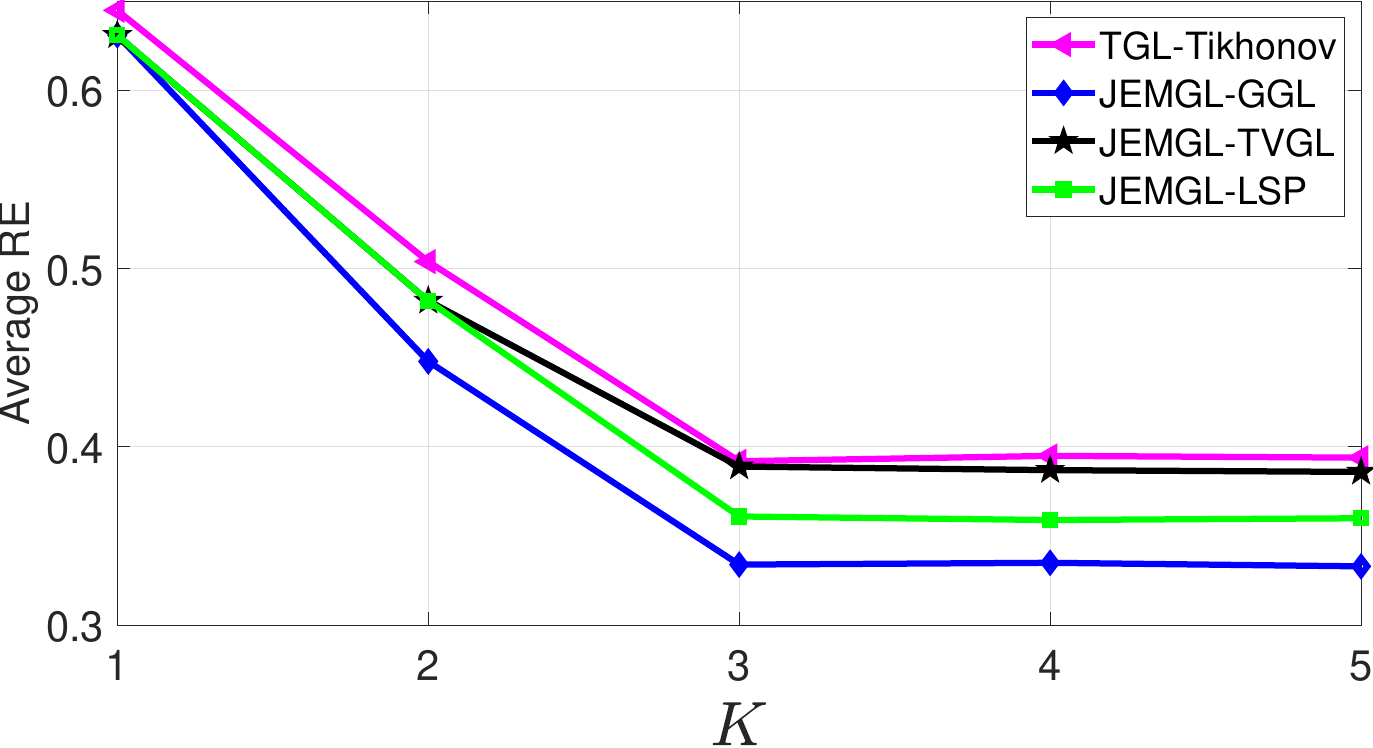}
}
\hspace{0.2ex}
\subfigure[FS]{
\includegraphics[width=0.7\columnwidth, height=0.38 \columnwidth ]{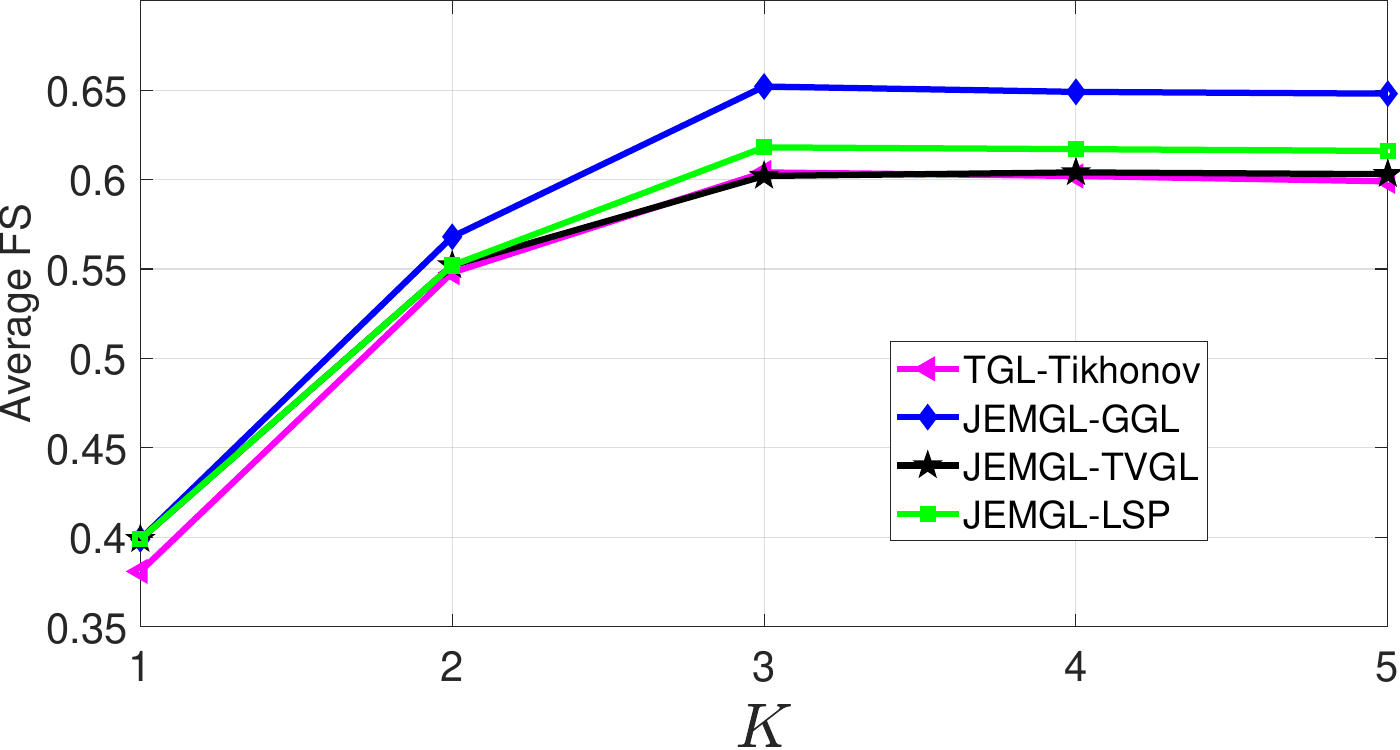}
}

\setlength{\belowcaptionskip}{-0.12 cm}
\caption{The effects of the number of graphs $(K)$ on the performance of the joint estimation methods. }
 \label{fig:simulation2}
\end{figure}

\begin{figure}[!t] 
\setlength{\belowcaptionskip}{-1cm}
\centering

\includegraphics[width=0.7\columnwidth , height=0.43 \columnwidth]{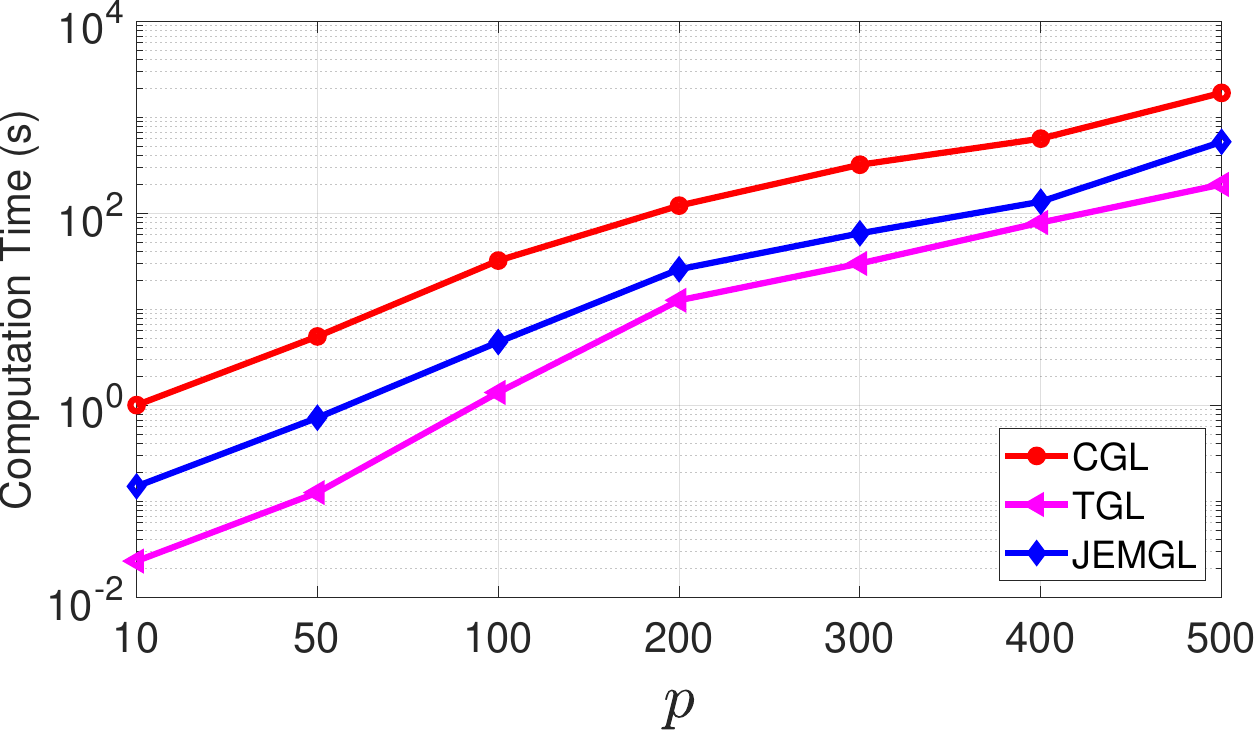}

\setlength{\belowcaptionskip}{-0.15 cm}
\caption{The comparison of computation time for different $p$.}
 \label{fig:simulation3}
\end{figure}

\subsubsection{Effect of sample size and signal dimension}
Now, we investigate the effect of the sample size $n$ and the signal dimension  $p$ (i.e., the number of graph nodes) on the performance of all the methods. We generate signals from $K=3$ graphs with pattern 3 and obtain several observations with a balanced sample size. We set the dimension $p=30$ and vary the total sample size $n$ from $150$ to $4500$ to examine results from the high-dimensional to the low-dimensional settings. The results are presented in Fig. \ref{fig:simulation1}. Overall, the performance of all methods increases with the growing $n/p$, which is consistent with the results in Theorem $\ref{theorem:mainresults}$. Specifically, the JEMGL-LSP outperforms other methods at all settings. Besides, all the joint estimation methods achieve much better performance than that of the separate estimation method (CGL-Lasso), especially in the small regime of $n/p$. These results demonstrate that the proposed penalty can substantially improve the graph estimation accuracy, especially in high-dimensional settings.

\subsubsection{Effect of the number of graphs} We also study the effect of the number of graphs $K$ on the performance of the joint estimation methods. In this experiment, we construct graphs with pattern 1 and vary the number of graphs from $1$ to $5$. We set the graph nodes $p=30$ and generate $n_k=300$ signal samples from each graph $G_k$. The results are presented in Fig. \ref{fig:simulation2}. In general, we observe that at first the performance of all methods increases significantly with the growing $K$, but then varies smoothly or even decreases slightly when $K\geq 4$. The intuitive explanations behind this phenomenon are given as follows. On one hand, separate estimation ($K=1$) is suboptimal due to the lack of additional information from related networks, whereas the joint estimation can achieve better performance by taking advantage of similarity among related works ($K>1$). On the other hand, as shown in Theorem $\ref{theorem:mainresults}$, with a fixed sample size, the estimation accuracy will decrease with the growing $K$. This is because, there are more unknown parameters with larger $K$, leading to the multiple graph estimation tasks more challenging. Therefore, a proper $K$ needs to be chosen to achieve a certain performance with fixed sample size.

\subsubsection{Computation time}
In Fig. \ref{fig:simulation3}, we compare the computation time of the proposed JEMGL algorithm with the two baseline methods. As the value of the Gram matrix $\tilde{\mathbf{J}}$ does not affect the computation time, hence we just set $\tilde{\mathbf{J}}=\mathbf{I}$ and omit the regularization indicator for brevity. All the methods are implemented in MATLAB R2019b and run on a 2.6-GHz Intel Core processor with 32-GB RAM. In the experiments, we generate signal samples from $K=3$ graphs with pattern 1 for different number of nodes $p=\{10,50,100,200,300,400,500\}$. As shown in Fig. \ref{fig:simulation3}, the proposed JEMGL algorithm provides faster convergence than the CGL method. This is due to the close-form solutions we derived in Section \ref{sec:optimization} for each of the ADMM subproblems. Although the computation complexity per iteration in both CGL and JEMGL are $\mathcal{O}(p^3)$, the CGL needs to solve nonnegative quadratic programs which require several number of iterations to converge. Note that the JEMGL requires longer computation time than the TGL, because the computation complexity of TGL is $\mathcal{O}(p^2)$. However, the JEMGL can be applied in many different applications, thus the choice of algorithms needs to make a trade-off between the computation complexity and the generality.

\subsection{Experiments on Real Datasets}
In this subsection, we apply the JEMGL algorithm on two real-world datasets to demonstrate that our joint graph estimator can be flexibly used to find meaningful network topologies from heterogeneous graph signals in different scientific fields. 
\subsubsection{Applications in Webkb Data}

\begin{figure*}[t!] 
\setlength{\belowcaptionskip}{-0.5cm}
\centering
\subfigure[Common Structure]{
\includegraphics[width=0.8 \columnwidth, height=0.7  \columnwidth]{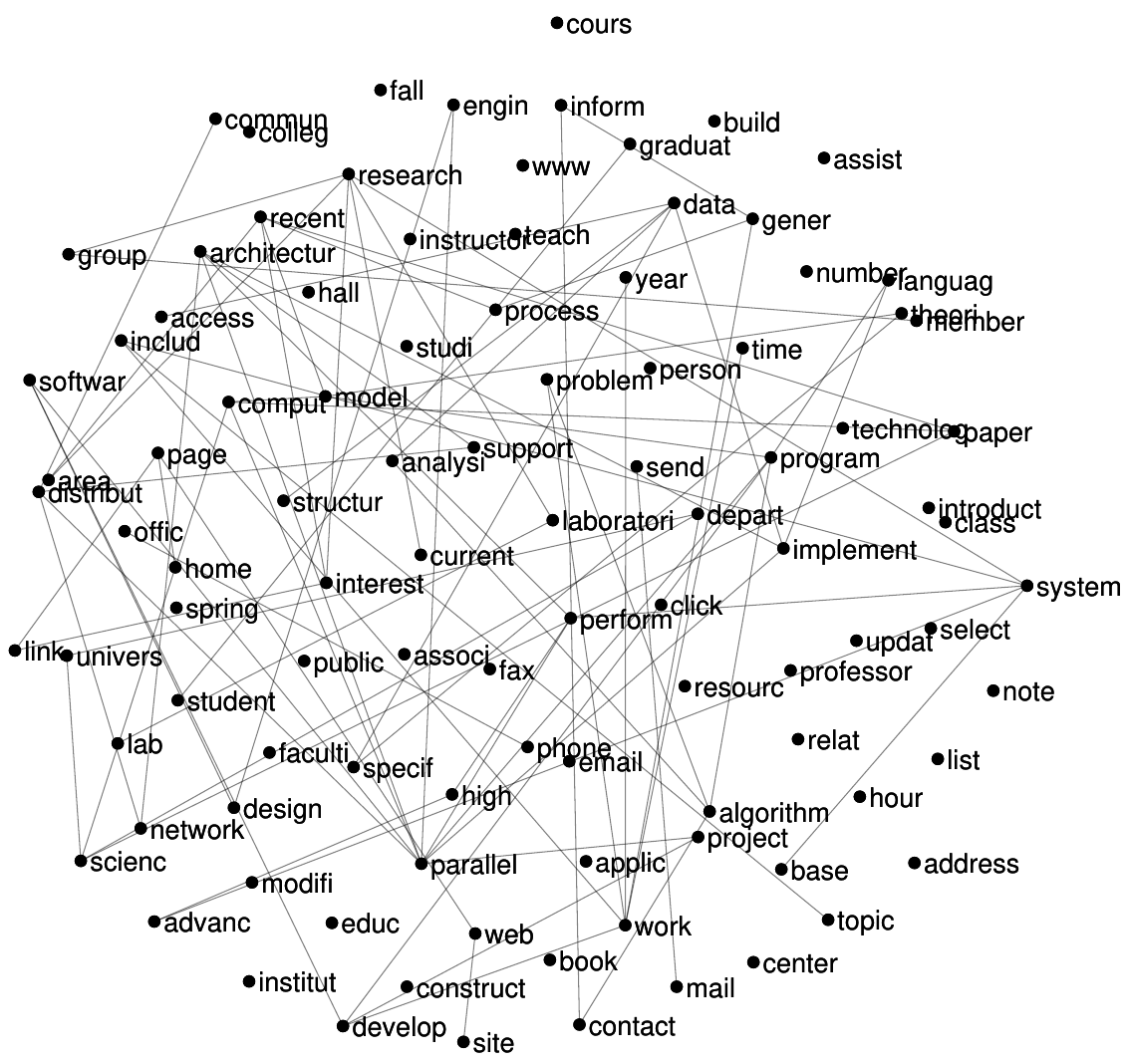}
}
\hspace{0.2ex}
\subfigure[`Course' ]{
\includegraphics[width=0.8\columnwidth, height=0.7  \columnwidth]{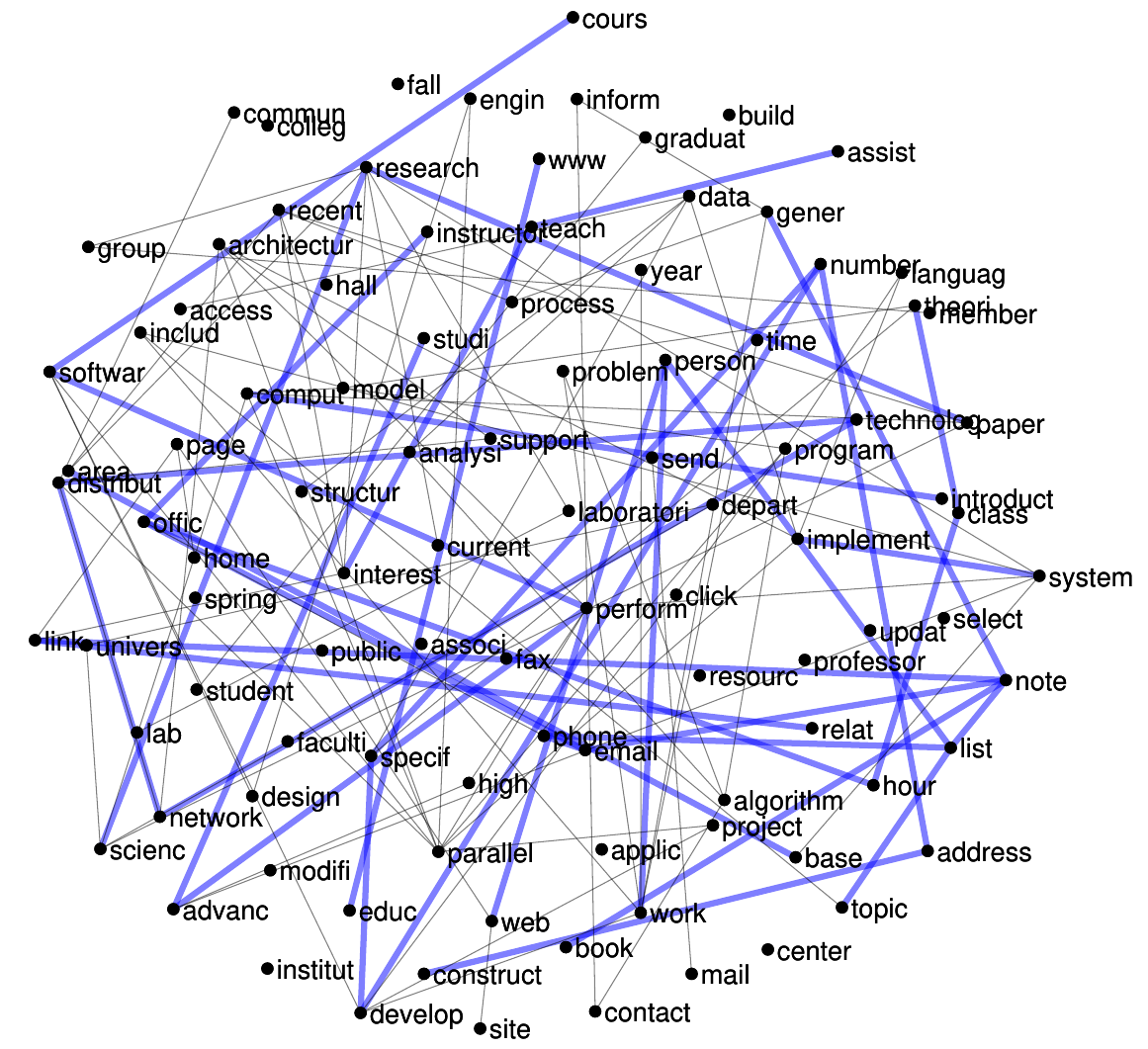}
}
\hspace{0.2ex}
\subfigure[`Student']{
\includegraphics[width=0.8\columnwidth, height=0.7  \columnwidth]{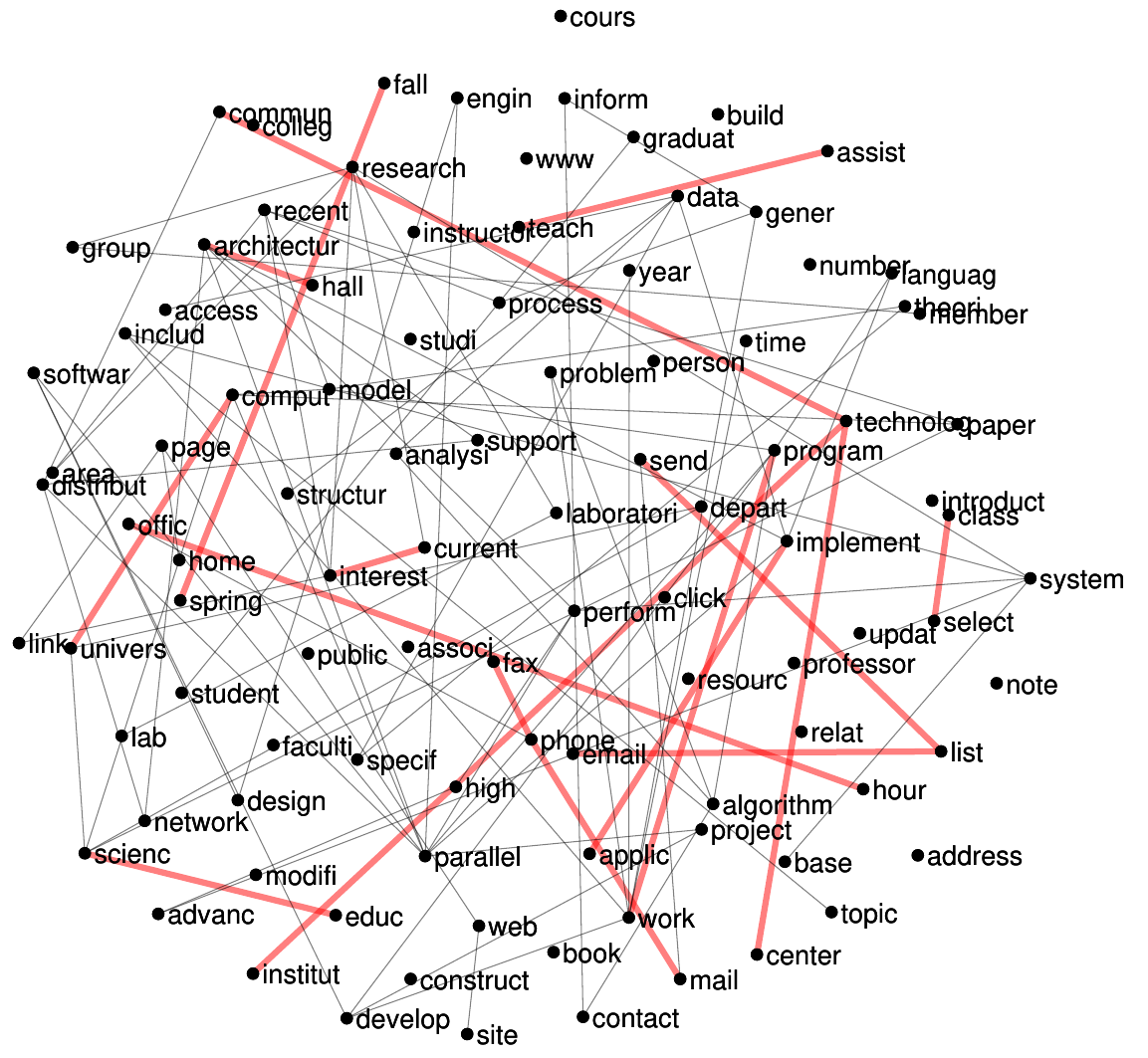}
}
\hspace{0.2ex}
\subfigure[`Faculty' ]{
\includegraphics[width=0.8\columnwidth, height=0.7  \columnwidth]{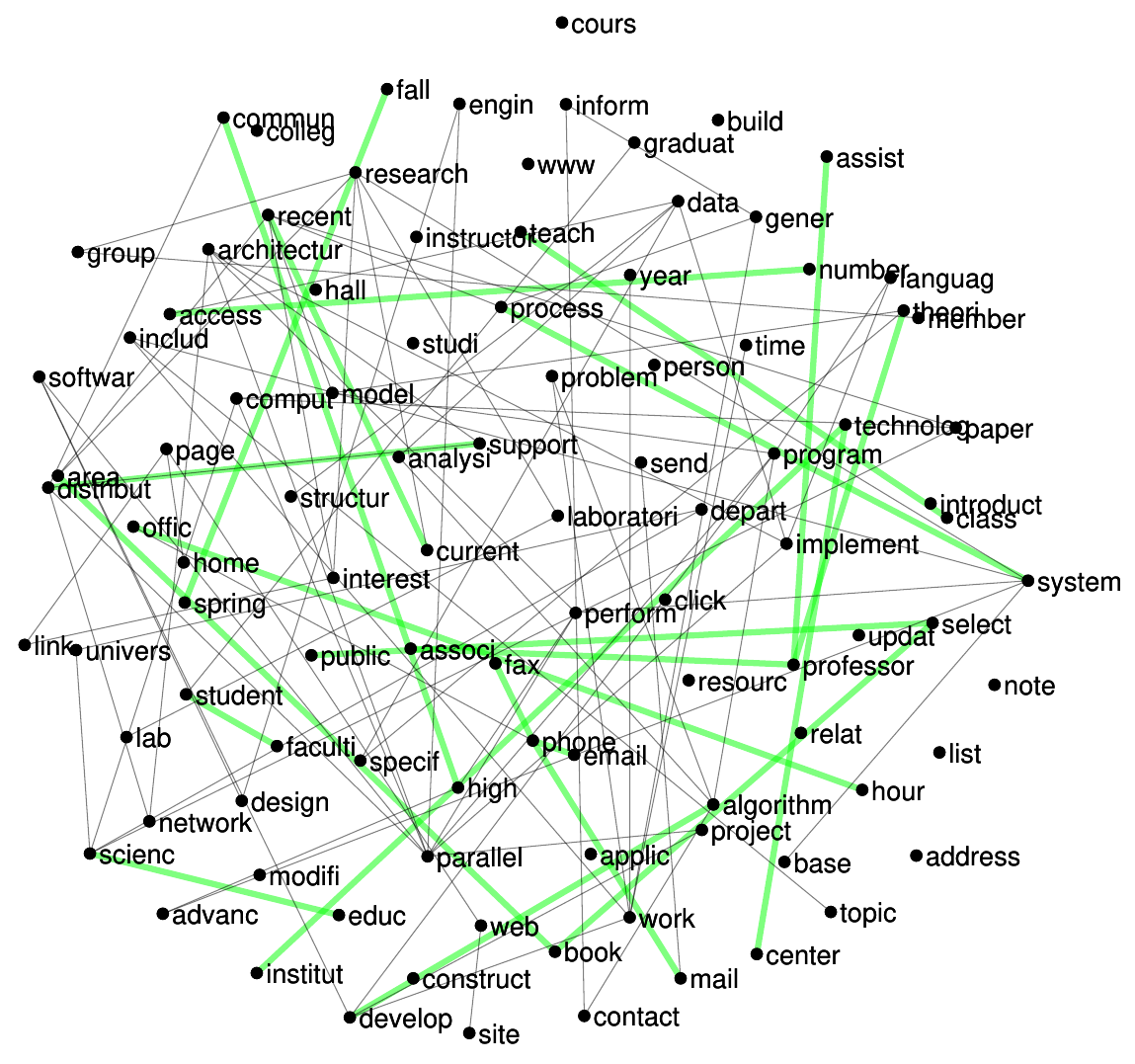}
}

\setlength{\belowcaptionskip}{-0.2cm}
\caption{The visualization of the inferred word-word networks by the JEMGL-LSP on \textbf{Webkb} dataset. The nodes represent the selected 100 words. The edges denote the  conditional dependence relations among words. The black lines are the edges shared in
all three categories. The color lines are the unique edges in some categories.  (a) The common structure among all networks. (b) The inferred word-word network for `Course' webpages. (c) The inferred word-word network for `Student' webpages. (d) The inferred word-word network for `Faculty' webpages. }
\label{fig: realworld}
\end{figure*}

We analyse the \textbf{Webkb} data set\footnote{The full data set can be downloaded from the machine
learning repository at the University of California, Irvine, http://www.
ics.uci.edu/~mlearn/.} from the World Wide Knowledge Base project at Carnegie Mellon University \cite{Cravenweb}. The data set contains webpages from websites at computer science departments in various universities. The webpages include seven categories: student, faculty, course, project, staff, department, and the other. We aim to construct the word-word networks for each category. In the experiment, only 1228 webpages ($n=1228$) corresponding to the three largest categories are
selected: student (544 webpages), faculty (374 webpages) and course (310 webpages). The original data set was preprocessed by Cardoso-Cachopo \cite{cachopo2007improving}. The log-entropy weighting method \cite{dumais1991improving} was used to calculate the word-document matrix $\mathbf{X}\in \mathbb{R}^{p \times n}$ with $p$ and $n$ denoting the number of distinct words and webpages. In particular, let $f_{ij}, i \in [p], j \in [n]$ be the number of times the $i$-th word appears in the $j$-th webpage and let $h_{ij}=\frac{f_{ij}}{\sum_{j=1}^n f_{ij}}$. Then, the log-entropy weight of
the $i$-th word is defined as  $e_{i}=1+\sum_{j=1}^{n} h_{i, j} \log \left(h_{i, j}\right) / \log (n)$. Finally, the $(i,j)$-th entry of the word-document matrix is given by $[\mathbf{X}]_{ij}=e_i\log(1+f_{ij}), i \in [p], j\in [n]$, and it is normalized along each row. In our experiment, $p=100$ words with the highest log-entropy weights are selected out of a total of $4800$ words. Since we are trying to learn networks that belong to different classes, we apply the JEMGL-LSP method on the word-document matrix $\mathbf{X}$ to learn three word-word networks.

The resulting word-word networks for each category are shown in Fig. \ref{fig: realworld}.  In Fig. \ref{fig: realworld}, the black lines
are links shared in all categories, and the color lines are uniquely presented in some categories. The resulting common network structure is shown in Fig. \ref{fig: realworld}(a), the estimated networks for the course, student, and faculty categories are shown in Fig. \ref{fig: realworld}(b), (c), and (d), respectively. Clearly, most edges are black lines, which represents the homogeneity component shared in all categories. As an example, some standard phrases in computer science, such as comput-scienc, softwar-develop, program-language, and web-page, etc., are significant across all the tree categories and can be found in Fig. \ref{fig: realworld}(a). 

The JEMGL-LSP also allows us to explore the heterogeneity between different categories. For instance, some links uniquely appear in 'Course' graph, such as theori-class, class-hour, and soft-cours. As these are course-related terms, it is reasonable to expect these links to appear in the course category. Similarly, 
it can be seen that the words pair select-class is only linked in the student category, since graduate students have to choose classes. On the other hand, some word pairs only have links in the faculty category, such as assist-professor and associate-professor. In addition, teach-assist is shown frequently in the course and student category but less likely to appear in faculty class. 
Overall, this experimental results demonstrate that the JEMGL-LSP method can not only capture the basic common semantic structure of the webpages, but also identifies meaningful differences across the various categories.

\subsubsection{Applications in ECoG time series}
In epilepsy study, inferring the connectivity of brain regions from the electrocorticogram (ECoG) time series is an informative way to understand the mechanisms underlying epilepsy seizures \cite{kramer2008emergent}. As a second case study, we apply the proposed JEMGL-TVGL method to infer a time-varying graph from a publicly available ECoG dataset \footnote{The ECoG data is found at http://math.bu.edu/people/kolaczyk/datasets.html. Eight sets of measurements are provided in the file, in our experiment, measurements recorded during the seventh seizure are used.} to explore the time-varying brain connectivity over the period of a seizure. This ECoG data were recorded via 76 electrodes implanted in an epilepsy patient's brain (an 8x8 electrode grid located at the cortical level and two 6-electrode strips placed deeper in the brain). The combined electrode network recorded $p=76$ time series, which consisted of voltage levels measured in a brain region within proximity of each electrode. Samples of data before and after a seizure onset (i.e., the so-called pre-ictal and ictal periods) were then extracted from the ECoG time series. The sampling rate was 400 Hz and each period lasted for 10 seconds. In order to ensure that the signal is approximately stationary, each of the 10s intervals was divided into 10 successive segments, each spanning 1s. Therefore, we have $K=20$ segments in total and the sample size is $n_k=400, \forall k=1,\ldots,20$. 
\begin{figure}[t!] 
\setlength{\belowcaptionskip}{-0.5cm}
\centering
\includegraphics[width=0.9\columnwidth]{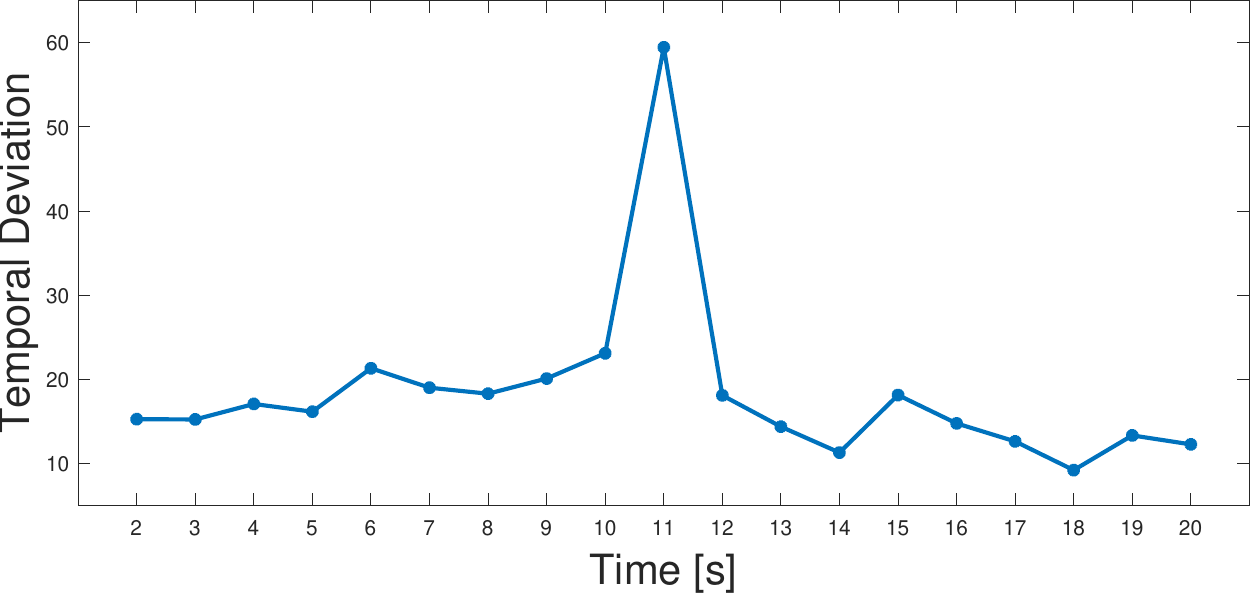}
\setlength{\belowcaptionskip}{-0.2cm}
\caption{Plot of the temporal deviation of the inferred brain connectivity networks from ECoG time series, which detects a structure shift at seizure onset.
  }
\label{fig: EoCG-dev}
\end{figure}
The data are arranged so that the first $10$ segments belong to the pre-ictal period and the remaining $10$ segments come from the ictal period. As examined in \cite{WANG20141744}, we would expect the brain connectivity network to vary smoothly within the pre-ictal period and the ictal period, but there would be a sudden shift in the network structure caused by the epileptic seizure. Therefore, we set all the weighted parameters in the time-varying graph lasso penalty to $1$ except the $w_{11,10}=10$. 
\begin{figure*}[t!] 
\setlength{\belowcaptionskip}{-0.5cm}
\centering
\subfigure['Pre-ictal']{
\includegraphics[width=0.65 \columnwidth,  height=0.52 \columnwidth]{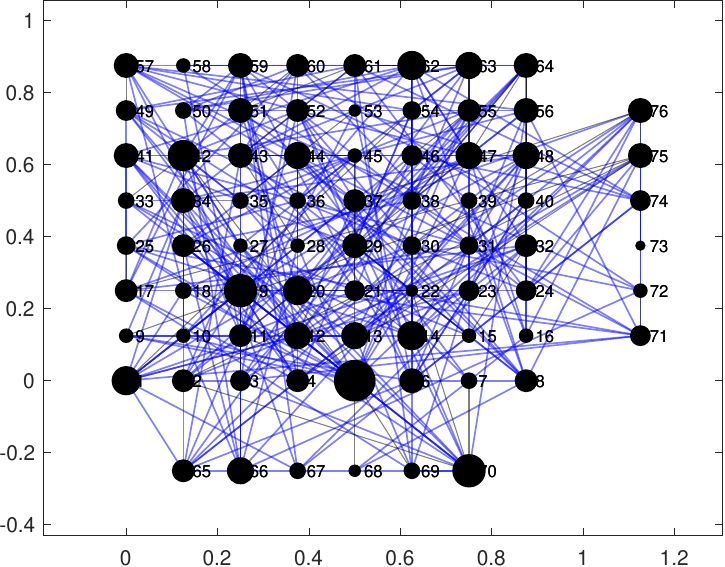}
}
\hspace{0.2ex}
\subfigure[`Ictal']{
\includegraphics[width=0.65\columnwidth, height=0.52 \columnwidth]{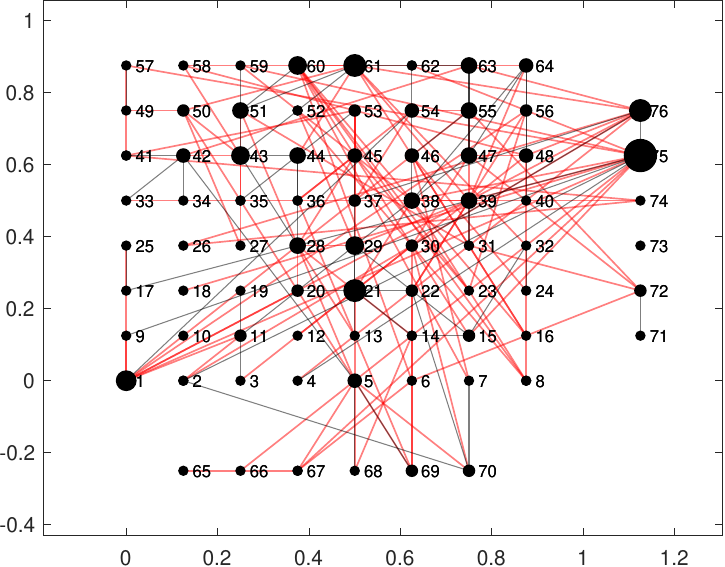}
}
\hspace{0.2ex}
\subfigure[Degree difference]{
\includegraphics[width=0.6\columnwidth, height=0.52 \columnwidth ]{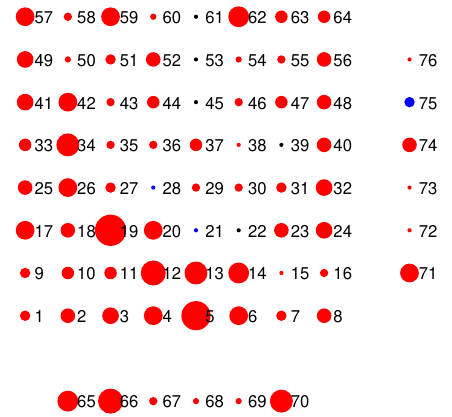}
}
\setlength{\belowcaptionskip}{-0.2cm}
\caption{The visualization of brain connectivity networks inferred by the JEMGL-TVGL on the \textbf{ECoG} time series. The nodes represent the 76 electrodes implanted in the brain of an epilepsy patient. The edges represent the connectivity between brain regions about each electrode, where the black lines are shared in both two periods and color links are unique in each period. The node degrees of inferred networks are encoded by circle radii. (a) Network inferred at pre-ictal interval; (b) Network inferred at ictal onset; (c) The node degree changes at seizure onset, where the circle radius indicates the magnitude of the degree difference and the color denotes the sign: a red (blue) circle denotes a decrease (increase) in degree.}
\label{fig: ECoGvisu}
\end{figure*}

After running our JEMGL-TVGL method, we plot the temporal deviation $\|\mathbf{L}_k-\mathbf{L}_{k-1}\|_{\text{F}}$, in Fig. \ref{fig: EoCG-dev}. We note that there is a spike in the temporal deviation score at seizure onset. This indicates a  significant shift in network structure at this specific time. We show the two inferred networks during this shift in Fig. \ref{fig: ECoGvisu}(a) and Fig. \ref{fig: ECoGvisu}(b). We observe a global decrease in the level of brain connectivity after the onset of the epileptic seizure, i.e., there are fewer edges in Fig. \ref{fig: ECoGvisu}(b) than in Fig. \ref{fig: ECoGvisu}(a). To quantify these changes, we employ the node degree, which is indicative of how well-connected a given node is, as a measure for network analysis. The value of the node degree is encoded by the radii of circles associated with the nodes in Fig. \ref{fig: ECoGvisu}(a) and Fig. \ref{fig: ECoGvisu}(b). We show how the degree of each node changes during a seizure in Fig. \ref{fig: ECoGvisu}(c). We observe that the degree of most nodes decrease after the onset of the seizure. These observations provide empirical evidence that the brain network becomes less connected, meaning that the diffusion of information after a seizure is inhibited. Such observations are consistent with the findings in \cite{kramer2008emergent}. We also find that a depth electrode exhibits an increase in node degree after seizure onset, the brain region around this electrode may be a potential target for further study on epilepsy.

In summary, the above two real-world case studies demonstrate the flexibility of the proposed JEMGL algorithm in learning different types of networks in different applications.

\section{Conclusion} \label{sec:conclustion}
In this paper, we have presented a general graph estimator for jointly estimating multiple graph Laplacians from heterogeneous graph signals, called JEMGL. The JEMGL is capable of capturing sparse, heterogeneity and homogeneity components among graphs via structured fusion regularization. The proposed regularization term generalizes useful convex penalties for joint estimation of multiple graphs, allowing us to encode different topological patterns among graphs, and can be particularly useful in learning from heterogeneous graph signals. In the JEMGL, we have developed an ADMM-based algorithm to solve the penalized graph estimator efficiently. Moreover, we have provided theoretical guarantees for the JEMGL, which establishes a non-asymptotic estimation error bound under the high-dimensional setting and also enables us to investigate the relationship between the estimation accuracy with the topological patterns, the number of data samples, and the number of networks. Experimental results on synthetic and real data have demonstrated the superior performance of the JEMGL.

Throughout this paper, we assume that the signal labels are known. It is of great interest to study extensions of the JEMGL to learn graph Laplacian  mixture model, which are left for our future research.

\balance

\bibliographystyle{IEEEtran}
\bibliography{bib/StringDefinitions,bib/IEEEabrv,bib/bscoop}

\newpage
\begin{appendices}
\section{Proof of Theorem \ref{theorem:mainresults}} \label{sec:prooftheorem}
We now turn to the proof of theorem \ref{theorem:mainresults}.
To treat multiple graph Laplacians in a unified way, the parameter space $\Omega$ is defined to the set of $\mathbb{R}^{(p K) \times (p K) }$ symmetric block diagonal matrices, where each diagonal block corresponding a graph Laplacian matrix. In this parameter space, we define a map: $f: \mathbb{R}^{(pK) \times (pK) } \rightarrow \mathbb{R}$, given by
\begin{equation*}
f(\mathbf{\Delta})=-\mathcal{F}_n(\mathbf{L}^*+\mathbf{\Delta})+\mathcal{F}_n(\mathbf{L}^*)+\rho_n\big(\mathcal{P}(\mathbf{L}^*+\mathbf{\Delta})-\mathcal{P}(\mathbf{L}^*)\big).
\end{equation*}
This map provides information
on the behavior of our objective function in the neighborhood of $\mathbf{L}^*$. Based on the optimality of the solution, we have $f(\hat{\mathbf{\Delta}}_n) \leq 0$, where $\hat{\mathbf{\Delta}}_n=\hat{\mathbf{L}}_{\rho_n}-\mathbf{L}^*$. Our goal is to obtain an upper bound of $\|\hat{\mathbf{\Delta}}_n\|_{\mathrm{F}}$, which depends on the properties of the empirical loss function $\mathcal{F}_n(\cdot)$ and penalty function $\mathcal{P}(\cdot)$. In the following, we first present some lemmas, and then establish our main results based on these lemmas.

Denote the support space of $\mathbf{L}_{k}^*$ as $\mathcal{S}_k=\{(i,j): [\mathbf{L}_{k}^*]_{ij}\neq 0, i,j=1,\cdots,p,i \neq j\}$, then the support space of $\mathbf{L}^{*}$ is $\mathcal{S}=\bigcup_{k=1}^{K} \mathcal{S}_{k}$.
The orthogonal complement of support space $\mathcal{S},$ namely, is defined as the set
\begin{equation}
\mathcal{S}^{\perp}:=\left\{\mathbf{L}^{\prime} \in \Omega \mid \left\langle \mathbf{L}, \mathbf{L}^{\prime}\right\rangle=0, \forall  \mathbf{L} \in \mathcal{S}\right\}.
\end{equation}
Given a matrix set $\mathbf{L} \in \Omega$, we use $\mathbf{L}_{\mathcal{S}}$ to denote the projection of $\mathbf{L}$ onto $\mathcal{S}$. 

\begin{lemma} \label{lemma:lemma1} Properties of $\mathcal{P}(\cdot)$:

(i) Our penalty defined in (\ref{eq:pengl}) is a seminorm, convex and decomposable with respect to ($\mathcal{S}, \mathcal{S}^\perp$), i.e., 
\begin{equation}
\mathcal{P}(\mathbf{L}_1+\mathbf{L}_2)=\mathcal{P}(\mathbf{L}_1)+\mathcal{P}(\mathbf{L}_2),\forall  \mathbf{L}_1 \in \mathcal{S},\mathbf{L}_2 \in \mathcal{S}^\perp.
\end{equation} 
Besides,
\begin{equation} \label{eq:ieqpenalty}
\begin{aligned}
\mathcal{P}(\mathbf{L}^*+\mathbf{\Delta})-\mathcal{P}(\mathbf{L}^*)
\geq \mathcal{P}\left(\mathbf{\Delta}_{\mathcal{S}^{\perp}}\right)-\mathcal{P}\left(\mathbf{\Delta}_{\mathcal{S}}\right).
\end{aligned}
\end{equation}

(ii) The dual norm of $\mathcal{P}(\mathbf{L})$ represented by $\mathcal{P}^{*}(\mathbf{L})$ can be bounded by
\begin{equation} \label{eq:penalitybound}
\mathcal{P}^{*}(\mathbf{L})\leq\left(1+\sigma_{\min}(\widetilde{\mathbf{J}})\sqrt{K}\right)\max_{k} \|\mathbf{L}_k\|_{\max,\mathrm{off}}.
\end{equation}

(iii) For  $\mathbf{L} \in \mathcal{S}$, $\mathcal{P}(\mathbf{L})$ satisfies the following inequality: 
\begin{equation} \label{eq:39}
\mathcal{P}(\mathbf{L}) \leq \sqrt{s}\Big(1+\rho \sqrt{\sigma _{\max }(\widetilde{\mathbf{J}})}\Big) \|\mathbf{L}\|_{\mathrm{F}},
\end{equation}
where $s$ representing the sparsity parameter, i.e., $s:=\mathrm{card}(\mathcal{S})$.
\end{lemma} 
\begin{proof}
See Appendix \ref{sec:proofoflemm1} for details.
\end{proof}

\begin{lemma} \label{lemma:lemma2}
Properties of $\mathcal{F}_n(\cdot)$:

(i) The gradient of $\mathcal{F}_n(\mathbf{L}^*)$ is a block matrix with the $k$-th block is given by
\begin{equation}
[\nabla \mathcal{F}_n(\mathbf{L}^*)]_k=\frac{n_k}{n}\left(\mathbf{\Sigma}_k^*+\mathbf{M}- \hat{\boldsymbol{\Sigma}}_k \right),
\end{equation}
where $\mathbf{\Sigma}_k^*=(\mathbf{L}_k^*)^{\dagger}$ and $\mathbf{M}=\frac{1}{p}\mathbf{1}\mathbf{1}^{\top}$.

(ii) Let $\frac{n_k}{n}>\tau>0$ for all $k$ and $n \geq \frac{2}{\tau} \ln p$, the $\mathcal{P}^*(\cdot)$ norm of the gradient is bounded by 
\begin{equation}
\mathcal{P}^* \big (\nabla \mathcal{F}_n\left(\mathbf{L}^*\right)\big ) \leq \gamma_n,
\end{equation}
with probability a least $1-\frac{2K}{p}$, and $\gamma_n= \left(1+ \sigma_{\min}(\widetilde{\mathbf{J}})\sqrt{K}\right)\left(\frac{1}{p}+ 40 \sqrt{2} \max_{k,i}[\mathbf{\Sigma}_k^*]_{ii}\sqrt{\frac{\ln p}{n \tau}}\right)$.

(iii) (Restricted curvature conditions) Let $c$ be a universe constant, and for $\|\mathbf{\Delta}\|_{\mathrm{F}} \leq r$, $\min_k (\frac{n_k}{n}) \geq \tau>0$ and $\lambda_{\mathbf{L}} \equiv \max_k \|\mathbf{L}_k^*\|_{\mathrm{2}} $, 
\begin{equation} \label{eq:RSCofLoss}
\begin{aligned}
& -\mathcal{F}_{n}\left(\mathbf{L}^*+\mathbf{\Delta}\right)+\mathcal{F}_{n}\left(\mathbf{L}^{*} \right)+\langle \nabla \mathcal{F}_n(\mathbf{L}^*),\mathbf{\Delta} \rangle\\
\geq &\frac{ \tau}{2(\lambda_{\mathbf{L}}+r)^2} \|\mathbf{\Delta}\|_{\mathrm{F}}^{2} .
\end{aligned}
\end{equation}

\end{lemma}
\begin{proof}
See Appendix \ref{sec:proofoflemm2} for details.
\end{proof}

\begin{lemma} \label{lemma: keylemmea}
Define a cone $\mathbb{C}=\{\mathbf{\Delta} \in \mathbb{R}^{pK \times pK }: \mathcal{P}(\mathbf{\Delta}_{\mathcal{S}^{\perp}}) \leq 3 \mathcal{P}(\mathbf{\Delta}_{\mathcal{S}}) \}$.
Suppose the tuning regularization parameter $\rho_n \geq 2 \gamma_n$. Let $0 <\epsilon \leq r$, if $f(\mathbf{\Delta}) >0$ for all elements $\mathbf{\Delta} \in \mathbb{C} \cap  \{\|\mathbf{\Delta}\|_{\mathrm{F}}=\epsilon\}$, then $\| \hat{\mathbf{\Delta}}_n\|_{\mathrm{F}}\leq \epsilon $.
\end{lemma}
\begin{proof}
See Appendix \ref{sec:proofoflemm3} for details.
\end{proof}

Now, we apply Lemma 1-3 to derive the non-asymptotic estimation error bound. We first compute a lower bound for $f(\mathbf{\Delta})$.  For an arbitrary $\mathbf{\Delta} \in \mathbb{C} \cap  \{\|\mathbf{\Delta}\|_{\mathrm{F}}=\epsilon\}$, by (\ref{eq:ieqpenalty}) and (\ref{eq:RSCofLoss} ), we have 
\begin{equation}
\begin{aligned}
 f(\mathbf{\Delta}) 
 \geq & - \langle\nabla \mathcal{F}_n(\mathbf{L}^*),\mathbf{\Delta} \rangle+\frac{ \tau}{2(\lambda_{\mathbf{L}}+r)^2} \|\mathbf{\Delta}\|_{\mathrm{F}}^{2}\\
 &+\rho_n \big(\mathcal{P}\left(\mathbf{\Delta}_{\mathcal{S}^{\perp}}\right)-\mathcal{P}\left(\mathbf{\Delta}_{\mathcal{S}}\right)\big). 
\end{aligned}
\end{equation}

By the Cauchy-Schwarz inequality, we have $|\langle\nabla \mathcal{F}_n(\mathbf{L}^*),\mathbf{\Delta} \rangle| \leq \mathcal{P}^*\big(\nabla \mathcal{F}_n(\mathbf{L}^*)\big) \mathcal{P}(\mathbf{\Delta})$.  Assuming $\rho_n \geq 2\mathcal{P}^*(\nabla \mathcal{F}_n(\mathbf{L}^*))$, we conclude that $|\langle\nabla \mathcal{F}_n(\mathbf{L}^*),\mathbf{\Delta} \rangle| \leq \frac{\rho_n}{2}\big(\mathcal{P}\left(\mathbf{\Delta}_{\mathcal{S}^{\perp}}\right)+\mathcal{P}\left(\mathbf{\Delta}_{\mathcal{S}}\right)\big)$, and hence that 
\begin{equation} \label{eq:lowerboundoff}
\begin{aligned}
 f(\mathbf{\Delta}) 
 \geq &  \frac{\rho_n}{2}\big(\mathcal{P}\left(\mathbf{\Delta}_{\mathcal{S}^{\perp}}\right)-3\mathcal{P}\left(\mathbf{\Delta}_{\mathcal{S}}\right)\big)+\frac{ \tau}{2(\lambda_{\mathbf{L}}+r)^2} \|\mathbf{\Delta}\|_{\mathrm{F}}^{2}\\
\geq &- \frac{3\rho_n}{2}\mathcal{P}\left(\mathbf{\Delta}_{\mathcal{S}}\right)+\frac{ \tau}{2(\lambda_{\mathbf{L}}+r)^2} \|\mathbf{\Delta}\|_{\mathrm{F}}^{2}.
\end{aligned}
\end{equation}
By (\ref{eq:39}), we have that
\begin{equation} \label{eq:peanle}
\begin{aligned}
\mathcal{P}\left(\mathbf{\Delta}_{\mathcal{S}}\right)
\leq \sqrt{s}\Big(1+\rho \sqrt{\sigma _{\max }(\widetilde{\mathbf{J}})}\Big) \|\mathbf{\Delta}\|_{\mathrm{F}}.
\end{aligned}
\end{equation}
Substituting (\ref{eq:peanle}) into the lower bound (\ref{eq:lowerboundoff}), we obtain that 
\begin{equation} \label{eq:ineqefe}
f(\mathbf{\Delta}) 
 \geq - \frac{3\rho_n}{2}\sqrt{s}\Big(1+\rho \sqrt{\sigma _{\max }(\widetilde{\mathbf{J}})}\Big) \|\mathbf{\Delta}\|_{\mathrm{F}}+\frac{ \tau}{2(\lambda_{\mathbf{L}}+r)^2} \|\mathbf{\Delta}\|_{\mathrm{F}}^{2}.
\end{equation}
The right-hand side of the inequality (\ref{eq:ineqefe}) is a strictly positive definite quadratic form
in $ \|\mathbf{\Delta}\|_{\mathrm{F}}$, some
algebra shows that $f(\mathbf{\Delta}) >0$, as long as
\begin{equation}
\|\mathbf{\Delta}\|_{\mathrm{F}} \geq \frac{3\rho_n \sqrt{s}\Big(1+\rho \sqrt{\sigma _{\max }(\widetilde{\mathbf{J}})}\Big)(\lambda_{\mathbf{L}}+r)^2}{\tau} \equiv \epsilon_{r}.
\end{equation}

On the basis of Lemma \ref{lemma: keylemmea}, if we show that there exists a $r_0$ such that $\epsilon_{r_0} \leq r_0$, then we have  $\| \hat{\mathbf{\Delta}}_n\|_{\mathrm{F}}\leq \epsilon_{r_0} $.

Consider the inequality $(a+b)^2c \leq b, \forall a, b, c >0$, this inequality holds when $a=b$ and $bc\leq 1/4$. We apply the inequality above with $a=\lambda_{\mathbf{L}}$, $b=r_0$ and $c=\frac{3\rho_n \sqrt{s}\Big(1+\rho \sqrt{\sigma _{\max }(\widetilde{\mathbf{J}})}\Big)}{\tau}$. Combing $\rho_n =2\gamma_n$ with $bc\leq 1/4$ yields 
\begin{equation}
\begin{aligned}
n & \geq  \frac{2^{13} 15^2\lambda_{\mathbf{L}}^2\kappa_{\widetilde{\mathbf{J}}}^2 \nu^2}{\tau^3} s\ln p, \\
 \epsilon_{r_0}  & \leq 24 \kappa_{\widetilde{\mathbf{J}}}\lambda_{\mathbf{L}}^2\tau^{-3/2}\left(\frac{\sqrt{s}}{p}+40\sqrt{2} \nu \sqrt{\frac{s \ln p}{n}}\right),
\end{aligned}
\end{equation}
where $\kappa_{\widetilde{\mathbf{J}}}=\left(1+ \sigma_{\min}(\widetilde{\mathbf{J}})\sqrt{K}\right)\left(1+\rho \sqrt{\sigma _{\max }(\widetilde{\mathbf{J}})}\right)$ and $\nu=\max_{k,i}[(\mathbf{L}_k^*)^{\dagger}]_{ii}$. 

\section{ Proof of Lemmas} \label{app:lemmas}
\subsection{Proof of Lemma \ref{lemma:lemma1} }
\label{sec:proofoflemm1}

\begin{proof}

(i) The proposed penalty (\ref{eq:pengl}) does not penalize the diagonal elements, hence it is a seminorm. The convexity of $\mathcal{P}(\mathbf{L})$ comes from the convexity of $\ell_1$-norm and the weighted $\ell_2$-norm. The decomposability is obvious from the definition.  Let $\mathbf{L}^*+\mathbf{\Delta}=\mathbf{L}^*+\mathbf{\Delta}_{\mathcal{S}^{\perp}}+\mathbf{\Delta}_{\mathcal{S}}$, by the triangle inequality, we have $\mathcal{P}(\mathbf{L}^*+\mathbf{\Delta}) \geq \mathcal{P}(\mathbf{L}^*+\mathbf{\Delta}_{\mathcal{S}^{\perp}})-\mathcal{P}(\mathbf{\Delta}_{\mathcal{S}})$. Combining $\mathbf{L}^* \in \mathcal{S} $ with the decomposability of $\mathcal{P}(\cdot)$ yields
\begin{equation}
\begin{aligned}
\mathcal{P}(\mathbf{L}^*+\mathbf{\Delta}) -\mathcal{P}(\mathbf{L}^*) & \geq \mathcal{P}(\mathbf{L}^*+\mathbf{\Delta}_{\mathcal{S}^{\perp}})-\mathcal{P}(\mathbf{\Delta}_{\mathcal{S}})-\mathcal{P}(\mathbf{L}^*)\\
&=\mathcal{P}\left(\mathbf{\Delta}_{\mathcal{S}^{\perp}}\right)-\mathcal{P}\left(\mathbf{\Delta}_{\mathcal{S}}\right).
\end{aligned}
\end{equation}

(ii) Following the definition of the dual norm associated with decomposable penaltys in \cite{loh2020book}, we have 
\begin{equation}
\begin{aligned}
\mathcal{P}^{*}(\mathbf{L}):=\max \left( \max_{k} \|\mathbf{L}_k\|_{\max,\mathrm{off}}, \max_{i,j,i\neq j}\left(\mathbf{L}_{ij}^{\top}\widetilde{\mathbf{J}}^{-1}\mathbf{L}_{ij}\right)^{1 / 2}\right),
\end{aligned}
\end{equation}
Based on the equivalence of vector norms, we have
\begin{equation}
\begin{aligned}
& \max_{i,j,i\neq j}\left(\mathbf{L}_{ij}^{\top}\widetilde{\mathbf{J}}^{-1}\mathbf{L}_{ij}\right)^{1 / 2} \\
 \leq &
\sigma_{\min}(\widetilde{\mathbf{J}})  \max_{i,j,i\neq j}\left\|\mathbf{L}_{ij}\right\|_2 \\
 \leq & \sigma_{\min}(\widetilde{\mathbf{J}})\max_{i,j, i\neq j} \sqrt{K} \|(\mathrm{L}_{1,ij},\cdots,\mathrm{L}_{K,ij})^{\top}\|_{\infty}\\
 = & \sigma_{\min}(\widetilde{\mathbf{J}})  \sqrt{K} \max_{k \in [K]} \|\mathbf{L}_k\|_{\max,\mathrm{off}},
\end{aligned}
\end{equation}
where $\sigma _{\min }(\widetilde{\mathbf{J}})$ represents the smallest singular value of matrix $\widetilde{\mathbf{J}}$.
Therefore, the dual norm can be bounded by
\begin{equation}
\mathcal{P}^{*}(\mathbf{L})\leq\left(1+ \sigma_{\min}(\widetilde{\mathbf{J}})\sqrt{K}\right)\max_{k} \left(\|\mathbf{L}_k\|_{\max,\mathrm{off}}\right).
\end{equation} 

(iii)
By Cauthy Schwarz inequality and the concavity of the square root function, we have 
\begin{equation}
\begin{aligned}
\frac{\mathcal{P}(\mathbf{L})}{\|\mathbf{L}\|_{\mathrm{F}}}
&\leq  \sup_{\mathbf{L} \in \mathcal{S}  } \left( \frac{\mathcal{P}_1(\mathbf{L})}{\|\mathbf{L}\|_{\mathrm{F}}} + \frac{\mathcal{P}_2(\mathbf{L})}{\|\mathbf{L}\|_{\mathrm{F}}}  \right) \\
& \leq  \sup_{\mathbf{L} \in  \mathcal{S}  } \left(  \frac{\sum_{k=1}^K \|\mathbf{L}_k\|_{1,\mathrm{off}}}{\sqrt{\sum_{k=1}^K\|\mathbf{L}_k\|_{\mathrm{F},\mathrm{off}}^2}}+\frac{\rho \sum_{i \neq j} \sqrt{ \mathbf{L}^{\top}_{ij} \widetilde{\mathbf{J}} \mathbf{L}_{ij}}}{\sqrt{\sum_{k=1}^K\|\mathbf{L}_k\|_\mathrm{F,off}^2}} \right) \\
& \leq  \sqrt{s}+\rho\sup_{\mathbf{\mathbf{L}} \in  \mathcal{S} } \left(\frac{\sum_{i \neq j}\sqrt{\sigma _{\max }(\widetilde{\mathbf{J}})\|\mathbf{L}_{ij}\|_\mathrm{F}^2}}{\sqrt{\sum_{k=1}^K\|\mathbf{L}_k\|_\mathrm{F,off}^2}} \right) \\
& \leq \sqrt{s}\Big(1+\rho \sqrt{\sigma _{\max }(\widetilde{\mathbf{J}})}\Big).
\end{aligned} 
\end{equation}
\end{proof}
\subsection{Proof of Lemma \ref{lemma:lemma2} }
\label{sec:proofoflemm2}
\begin{proof}

(i) Rewrite empirical loss function 
\begin{equation}
\begin{aligned}
\mathcal{F}_n(\mathbf{L}^*) &=\frac{1}{n} \sum_{k=1}^{K} n_k \Big[ \log(|\mathbf{L}_k^*|_+)-\mathrm{tr} (\hat{\boldsymbol{\Sigma}}_k\mathbf{L}_k^*)\Big]\\
&=\frac{1}{n} \sum_{k=1}^{K} n_k \Big[ \log(\det(\mathbf{L}_k^*+\mathbf{M}))-\mathrm{tr} (\hat{\boldsymbol{\Sigma}}_k\mathbf{L}_k^*)\Big]
\end{aligned}
\end{equation}
with $\mathbf{M}=\frac{1}{p}\mathbf{1}\mathbf{1}^{\top}$. By taking derivatives blockwise, we have 
\begin{equation}
\begin{aligned}
[\nabla \mathcal{F}_n(\mathbf{L}^*)]_k& =\frac{n_k}{n}\left( (\mathbf{L}_k^*+\mathbf{M})^{-1}- \hat{\boldsymbol{\Sigma}}_k\right)\\
&=\frac{n_k}{n}\left((\mathbf{L}_k^*)^{\dagger}+\mathbf{M}-\hat{\boldsymbol{\Sigma}}_k\right).
\end{aligned}
\end{equation}

(ii) Based on the inequality (\ref{eq:penalitybound}) and $\tau \leq \frac{n_k}{n} \leq 1$, we have 
\begin{equation}
\begin{aligned}
&\mathcal{P}^{*}\left(\nabla \mathcal{F}_n(\mathbf{L}^*)\right) \\
\leq &  \left(1+ \sigma_{\min}(\widetilde{\mathbf{J}})\sqrt{K}\right)\max_{k} \|\mathbf{\Sigma}_k^*+\mathbf{M}- \hat{\boldsymbol{\Sigma}}_k \|_{\max,\mathrm{off}}\\
\leq &  \left(1+ \sigma_{\min}(\widetilde{\mathbf{J}})\sqrt{K}\right)\left(\frac{1}{p}+\max_{k} \|\mathbf{\Sigma}_k^*- \hat{\boldsymbol{\Sigma}}_k \|_{\infty}\right)
\end{aligned}
\end{equation}
Following the Lemma 1 in \cite{2011penalized}, we have 
\begin{equation}
\begin{aligned}
 \mathrm{Pr}\left(\|\mathbf{\Sigma}_k^*- \hat{\boldsymbol{\Sigma}}_k \|_{\infty} > \delta \right) 
 \leq  4 \exp \left(-\frac{n_k \delta^2}{2^55^2\max_{i}[\mathbf{\Sigma}_k^*]_{ii}^2}\right),
\end{aligned}
\end{equation}
for all $\delta \in (0, 40\max_{i}[\mathbf{\Sigma}_k^*]_{ii})$. By taking the union bound, for $n \geq \frac{2}{\tau} \ln p$, we have 
\begin{equation}
\max_k \|\mathbf{\Sigma}_k^*- \hat{\boldsymbol{\Sigma}}_k \|_{\infty} \leq 40 \sqrt{2} \max_{k,i}[\mathbf{\Sigma}_k^*]_{ii}\sqrt{\frac{\ln p}{n \tau}}, 
\end{equation}
with probability at least $1-\frac{2K}{p}$. Therefore, we obtain that 
\begin{equation}
\gamma_n= \left(1+ \sigma_{\min}(\widetilde{\mathbf{J}})\sqrt{K}\right)\left(\frac{1}{p}+ 40 \sqrt{2} \max_{k,i}[\mathbf{\Sigma}_k^*]_{ii}\sqrt{\frac{\ln p}{n \tau}}\right).
\end{equation}

(iii) Let $\alpha \in [0,1]$, for any $\|\mathbf{\Delta} \|_{\mathrm{F}} \leq r $, a Taylor-series expansion yields
\begin{equation}
\begin{aligned}
& -\mathcal{F}_{n}\left(\mathbf{L}^*+\mathbf{\Delta}\right)+\mathcal{F}_{n}\left(\mathbf{L}\right)+\langle\,\nabla \mathcal{F}_{n}\left(\mathbf{L}\right),\mathbf{\Delta}\rangle\\
=&\sum_k^K\frac{1}{2}\tr\left(\mathbf{\Delta}_k^{\top}  [\nabla^{2}\mathcal{F}_{n}\left(\mathbf{L}^{*} \right)]_k \mathbf{\Delta}_k\right).
\end{aligned}
\end{equation}
with $[\nabla^{2}\mathcal{F}_{n}\left(\mathbf{L}^{*} \right)]_k$ being $k$-th block of the Hessian matrix,  which is given by
\begin{equation} \label{eq:grd}
\begin{aligned}
[\nabla^{2}\mathcal{F}_{n}\left(\mathbf{L}^{*} \right)]_k 
 = -\frac{n_k}{ n} \Big[\left(\mathbf{L}_{k}^{*}+\alpha\mathbf{\Delta}_k\right)^{-1} \otimes\left(\mathbf{L}_{k}^{*}+\alpha\mathbf{\Delta}_k\right)^{-1}\Big].
\end{aligned}
\end{equation}
Thus, we have
\begin{equation}
\begin{aligned}
&  -\mathcal{F}_{n}\left(\mathbf{L}^*+\mathbf{\Delta}\right)+\mathcal{F}_{n}\left(\mathbf{L}\right)+\langle\,\nabla \mathcal{F}_{n}\left(\mathbf{L}\right),\mathbf{\Delta}\rangle\\
\geq &\sum_k^K \frac{1}{2} \sigma_{\min }\left(-[\nabla^{2}\mathcal{F}_{n}\left(\mathbf{L}^{*} \right)]_k\right)\|\text{vec}(\mathbf{\Delta}_k)\|_{2}^{2}\\
\geq &\sum_k^K \frac{\tau}{2} \frac{\|\mathbf{\Delta}_k\|_{\mathrm{F}}^{2}}{\left\|\mathbf{L}_{k}^{*}+\alpha\mathbf{\Delta}_k\right\|_{2}^{2}}\geq  \frac{ \tau}{2(\lambda_{\mathbf{L}}+r)^2} \|\mathbf{\Delta}\|_{\mathrm{F}}^{2},
\end{aligned}
\end{equation}
by using the fact that $\sigma_{\min}(\mathbf{A}^{-1} \otimes \mathbf{A}^{-1})=\frac{1}{\|\mathbf{A}\|_{2}^{2}}$ for any symmetric invertible matrix and the triangle inequality $\left\|\mathbf{L}_{k}^{*}+\alpha\mathbf{\Delta}_k\right\|_{2}^{2} \leq$ $\left(\left\|\mathbf{L}_{k}^{*}\right\|_{2}+r \right)^{2} \leq (\lambda_{\mathbf{L}} +r)^2.$ 
\end{proof}
\subsection{Proof of Lemma \ref{lemma: keylemmea} }
\label{sec:proofoflemm3}
\begin{proof}

By the inequality (\ref{eq:RSCofLoss}), we have 
\begin{equation} \label{eq:deln}
 -\mathcal{F}_{n}\left(\mathbf{L}^*+\hat{\mathbf{\Delta}}_n\right)+\mathcal{F}_{n}\left(\mathbf{L}^{*} \right)
\geq -|\langle \nabla \mathcal{F}_n(\mathbf{L}^*),\hat{\mathbf{\Delta}}_n \rangle |.
\end{equation}
Based on Lemma \ref{lemma:lemma2}(ii), the right hand side of the inequality (\ref{eq:deln})
is further bounded below by $-\frac{\rho_n}{2}\left(\mathcal{P}\left(\hat{\mathbf{\Delta}}_{n,\mathcal{S}^{\perp}}\right)+\mathcal{P}\left(\hat{\mathbf{\Delta}}_{n,\mathcal{S}}\right)\right)$.
Applying Lemma \ref{lemma:lemma1}(i) and the fact $f(\hat{\mathbf{\Delta}}_n) \leq 0$, we obtain
\begin{equation}
\frac{\rho_n}{2}\mathcal{P}\left(\hat{\mathbf{\Delta}}_{n,\mathcal{S}^{\perp}}\right)-\frac{3\rho_n}{2}\mathcal{P}\left(\hat{\mathbf{\Delta}}_{n,\mathcal{S}}\right) \leq 0,
\end{equation}
which implies $\hat{\mathbf{\Delta}}_{n} \in \mathbb{C}$.
The rest of the proof follows exactly as Lemma 4 in \cite{negahban2009unified}.
\end{proof}
\end{appendices}

\end{document}